\documentclass[twoside]{article}
\usepackage[totalheight=7in, totalwidth=5in,paperwidth=6.6in,paperheight=9in]{geometry}

\usepackage[numbers]{natbib}

\usepackage[pdfstartview=FitH,
            colorlinks,
           linkcolor=reference,
            citecolor=citation,
            urlcolor=e-mail]{hyperref}

\usepackage{amsmath,amscd,amssymb,amsthm,times,pslatex,sectsty,fancyhdr,indentfirst}
\usepackage[all,cmtip,matrix,arrow,curve]{xy}
\usepackage{tikz}

\theoremstyle{plain}
\newtheorem{theorem}{\normalfont\sffamily\bfseries\normalsize Theorem}[section]
\newtheorem{lemma}[theorem]{\normalfont\sffamily\bfseries\normalsize Lemma}
\newtheorem{proposition}[theorem]{\normalfont\sffamily\bfseries\normalsize Proposition}

\newtheorem{corollary}[theorem]{\normalfont\sffamily\bfseries\normalsize Corollary}

\theoremstyle{definition}
\newtheorem{definition}[theorem]{\normalfont\sffamily\bfseries\normalsize Definition}

\newtheorem{example}[theorem]{\normalfont\sffamily\bfseries\normalsize Example}

\theoremstyle{remark}
\newtheorem{remark}[theorem]{\normalfont\sffamily\bfseries\normalsize Remark}

\newcommand{\Le}{\leqslant}
\newcommand{\Ge}{\geqslant}

\DeclareMathOperator{\Z}{\mathbb Z}
\DeclareMathOperator{\trdeg}{\mathrm tr.deg.}
\DeclareMathOperator{\Max}{\mathrm Max}
\DeclareMathOperator{\PSpec}{\mathrm PSpec}
\DeclareMathOperator{\ia}{\mathfrak a}
\DeclareMathOperator{\kk}{\mathds k}
\DeclareMathOperator{\kkk}{\mathbf k}
\DeclareMathOperator{\ib}{\mathfrak b}
\DeclareMathOperator{\p}{\mathfrak p}
\DeclareMathOperator{\q}{\mathfrak q}
\DeclareMathOperator{\m}{\mathfrak m}
\DeclareMathOperator{\Qt}{\mathrm Qt}
\DeclareMathOperator{\Id}{\mathrm Id}
\DeclareMathOperator{\V}{\mathbb V}
\DeclareMathOperator{\I}{\mathbb I}
\DeclareMathOperator{\Q}{\mathbb Q}
\DeclareMathOperator{\C}{\mathbb C}

\DeclareMathOperator{\Mn}{\bf M}
\DeclareMathOperator{\F}{\mathcal F}
\DeclareMathOperator{\Grp}{\mathcal G}
\DeclareMathOperator{\Orb}{Orb}

\RequirePackage{color}
\definecolor{todo}{rgb}{1,0,0}
\definecolor{answer}{rgb}{0,0,1}
\definecolor{new}{rgb}{1,0,1}
\definecolor{conditional}{rgb}{0,1,0}
\definecolor{e-mail}{rgb}{0,.40,.80}
\definecolor{reference}{rgb}{.20,.60,.22}
\definecolor{mrnumber}{rgb}{.80,.40,0}
\definecolor{citation}{rgb}{0,.40,.80}

\DeclareMathAlphabet{\mathpzc}{OT1}{pzc}{m}{it}
\RequirePackage{dsfont}

\DeclareMathOperator{\Aut}{Aut}

\DeclareMathOperator{\GL}{GL}

\DeclareMathOperator{\Spec}{Spec}

\DeclareMathOperator{\id}{id}

\DeclareFontEncoding{OT2}{}{}

\newcommand{\G}{\mathbf{G}}

\newcommand{\iso}{\cong}

\newcommand{\ZZ}{\mathds{Z}}

\newcommand{\Gm}{\mathds{G}_{m}}

\usepackage{abstract}
\sectionfont{\normalfont\sffamily\bfseries\normalsize\MakeUppercase}
\subsectionfont{\normalfont\sffamily\bfseries\normalsize}
\subsubsectionfont{\normalfont\sffamily\bfseries\normalsize}

\providecommand{\keywords}[1]{\footnotesize \textbf{\textit{Key Words:}} #1\normalsize}

\providecommand{\msc}[1]{\footnotesize \textbf{\textit{2010 Mathematics Subject Classification:}} #1\normalsize}

\title{\bfseries\sffamily\large GALOIS THEORY OF DIFFERENCE EQUATIONS WITH PERIODIC PARAMETERS\footnote{Address correspondence to Prof. Alexey Ovchinnikov, Department of Mathematics, CUNY Queens College, 65-30 Kissena Blvd., Queens, NY 11367, USA; E-mail: aovchinnikov@qc.cuny.edu}}

\author{\normalsize\bf Benjamin Antieau, Alexey Ovchinnikov,  and Dmitry Trushin \vspace{0.05in}  \\\it \small
 Department of Mathematics, University of California at Los Angeles, \\  \small
 \it  Los Angeles, California, USA \vspace{0.05in} \\ \it \small
Department of Mathematics, CUNY Queens College, 
Queens, New York, USA \\ \it \small
Department of Mathematics, CUNY Graduate Center, 
New York, New York, USA \vspace{0.05in}\\ \it \small
Einstein Institute of Mathematics,
The Hebrew University of Jerusalem,\\ \it\small
Jerusalem, Israel
\vspace{-0.2in}}

\usepackage{titling}

\date{}

\begin{document}
\maketitle

\begin{abstract}\footnotesize\textit {\textbf  {We develop a Galois theory for systems of linear difference equations with periodic parameters,
    for which we also introduce linear difference algebraic groups.
    We apply this to constructively test if solutions of linear $q$-difference equations, with $q \in \C^\ast$ and $q$ not a root of unity,
    satisfy any polynomial
    $\zeta$-difference equations with $\zeta^t=1$, $t \Ge 1$.}}
\end{abstract}
\smallskip


\medskip
\msc{Primary 12H10; Secondary 13N99, 20H25, 39A13.}
\bigskip
\setlength\parindent{0.7cm}
\section{Introduction} 
In this paper, we give a new Galois theory of systems of linear difference equations with periodic (of finite order) difference parameters. This appears
to be the first time that equations with difference parameters have been treated in the literature. In this theory,
the Galois groups are linear difference algebraic groups, and they measure difference algebraic dependence of solutions of difference equations.
Our Galois theory and Galois correspondence works over fields of any characteristic. In characteristic $p$, our Galois correspondence is presented here for separable extensions of difference
pseudofields\footnote{Our proofs apply to the non-separable case by using non-reduced Hopf algebras, as pointed out to us by Michael Wibmer.}.
Among numerous potential applications of our approach, we show how this can be applied to studying properties of solutions of $q$-difference equations.

For the purposes of this introduction, we briefly describe the setup and motivation in the following simple case.
Let $q \in \C\setminus \{0\}$. A $q$-difference equation of order $n$
is an equation in $f$ of the form
\begin{equation}\label{eq:nthqdiff}
f{\left(q^nz\right)} + a_{n-1}(z)\cdot f{\left(q^{n-1}z\right)}+\ldots+ a_0(z)\cdot f(z) = 0,
\end{equation}
where $a_0(z),\ldots,a_{n-1}(z) \in \C(z)$ are given. Over $\C(z)$, a solution to such an equation will not exist in general. However, there is a ring extension of $\C(z)$, called
the Picard--Vessiot ring for the equation that is universal for the property of having a full set of solutions to the equation. Let $f(z)$ be the solution in the Picard--Vessiot ring.
When, in addition, we fix a primitive $t$th root of unity, $\zeta$, and we let $\ZZ/(t)$ act on $\C(z)$ by $f(z)\mapsto f(\zeta t)$,
we show in this paper that we can construct a Picard--Vessiot extension and an action of $\ZZ/(t)$ on the ring extending
the action of $\zeta$ on $\C(z)$. The motivation for developing a Galois theory of difference equations with periodic parameters is to study the algebraic relations satisfied
by $$f(z),f(\zeta z),\ldots,f{\left(\zeta^{t-1}z\right)}.$$

In particular, as an application of the method of our
difference Galois groups with parameters,  Theorem~\ref{thm:6} gives an explicit,
complete description of all first-order $q$-difference equations
\begin{equation}\label{eq:diffeqintro}
    f\left(qz\right)=a(z)f(z)
\end{equation}
with rational coefficients whose solutions are $\zeta$-difference algebraically independent over the rational functions in variable
$z$ with coefficients belonging to the field $\kk$ of $q$-invariant meromorphic functions on $\C\setminus\{0\}$.
This description is easy to use: the inputs are simple functions in the multiplicities of the zeros and poles of $a(z)$.
Our proof requires a similar approach to that of \cite[Section~3]{CharlotteComp},
but it is substantially modified to take into account difference algebraic independence and make the result as explicit as possible.
As an example of our methods, we include a deduction of some algebraic independence properties of theta functions (see Theorem~\ref{thm:Benrelation}).

The approach of this paper resembles the Galois theory of difference equations with differential parameters studied
in~\cite{CharlotteArxiv,CharlotteComp,hardouin_differential_2008,CharlotteLucia11,CharlotteLucia12,CharlotteLucia13,CharlotteLuciaDescent,HDV,LuciaSMF},
where algebraic methods have been developed to test whether solutions of difference equations satisfy polynomial differential equations (see also \cite{Moshe2012} for a general Tannakian approach). In particular,
these methods can be used to prove H\"older's theorem which says that the $\Gamma$-function, which satisfies the difference equation $\Gamma(x+1)=x\cdot\Gamma(x)$,
satisfies no non-trivial differential equation over $\C(x)$.
However, when treating difference equations with differential parameters, one may use fields as the rings of constants.
This is not available when using difference parameters, as Example~\ref{ex:needtouserings} and \cite[Proposition~7.3]{Rosen} show.
The constants in our theory are rings that have zero-divisors, and this fact requires numerous additional subtleties into our approach. The key idea is to find a suitable notion of a difference closed ring.
We use the difference-closed pseudofields of \cite{Dima}, which we review in Section~\ref{sec:defs}. Another approach to the question of difference algebraic closure is in~\cite{Lando}, where difference
versions of valuation rings are given. However, since we require zero-divisors, Lando's approach is insufficient.

Picard--Vessiot extensions with zero divisors for systems of linear difference equation have been considered in~\cite{van_der_put_galois_1997,chatzidakis_definitions_2008,Morikawa}
with a non-linear generalization considered in \cite{Granier}. Also, Galois theories of linear difference
equations, without parameters, when the ground ring has zero divisors have been studied in
\cite{Andre,HopfGalois,Amano,Amano2,Wibmer}, where including zero divisors into the ground ring is needed and provides a much more transparent Galois correspondence.
In all the mentioned cases, the ground ring must be a finite product of fields (called Noetherian difference pseudofields).

Our approach allows us not only to treat parameters, but also prepares a foundation for studying the non-Noetherian case as we base our methods on a natural geometric approach to difference
varieties developed in \cite{Dima}, which has been further generalized to the non-Noetherian case in~\cite{DimaInfinite}.
However, to extend the theory to infinite parameter groups, it is necessary to treat the non-Noetherian case, as the same construction results in a non-Noetherian ring.
This generalization has been carried out in~\cite{OvWib} using the methods of~\cite{Wibmer,Wibmer2} and the present paper (see also \cite{DiVizioHardouinWibmer:DifferenceGaloisofDifferential} for the Galois theory of linear differential equations with difference parameters).
%We hope that this generalization, which would be of great interest in the study of $q$-difference equations, will be carried out in the near future using the methods developed in our paper
%as well as~\cite{Wibmer2}.

Some of our results can be treated in another way, via the method of faithfully flat descent
from algebraic geometry \cite{MilneEtale}. However, our theory gives a more flexible theory than that obtained
via descent, as explained in Section~\ref{sec:nonfaithful}.

The paper is organized as follows. We give basic definitions in Section~\ref{sec:basicdef}. The main properties of difference pseudofields are detailed in
Sections~\ref{sec:pseudofields} and~\ref{sec:pseudofieldsN}.
Section~\ref{sec:PVTheory} contains the development of our main technique, difference Galois theory (also called difference Picard--Vessiot theory) with periodic parameters.
Difference algebraic groups are introduced and studied in Section~\ref{sec:diffalggroups}.
We finish by showing in Section~\ref{sec:Applications} how to use our theory to study periodic difference algebraic dependencies among solutions of difference equations. In particular, we apply these results to study Jacobi's theta-function in Section~\ref{sec:Jacobi} and to give a complete characterization to all first-order $q$-difference equations with $\zeta$-difference algebraically independent solutions over rational functions in variable
$z$ with coefficients belonging to the field of $q$-invariant
meromorphic functions on $\C\setminus\{0\}$ in Section~\ref{sec:generalqdifference}.

\section{Basic definitions}\label{sec:defs}

\subsection{Difference rings}\label{sec:basicdef} Most of the basic notions on difference algebra can be found in~\cite{Cohn,levin_difference_2008}. Below, we will introduce those that we use here.
Let
$$
\Sigma_0 = \Z,\quad \Sigma_1 = \Z/t_1\Z\oplus\ldots\oplus\Z/t_s\Z,\quad \text{and}\quad \Sigma = \Sigma_0\oplus\Sigma_1,
$$
where each $t_i\Ge 2$. Let $\sigma$ be a generator of $\Sigma_0$ and
$\rho_i$, $1\Le i \Le s$, generate each component of $\Sigma_1$.

%\begin{definition}
A ring $R$ equipped with an action of a fixed subgroup $\Sigma'\subset\Sigma$ by automorphisms is called a $\Sigma'$-ring.
%\end{definition}

\begin{example}
Let $R = \C(x)$ and $\sigma(x) = px$, $\rho(x) = qx$ with $p,q \in \C^\ast$, $|p|\ne 1$ and $q$ a primitive $m$-th root of unity for some $m\Ge 2$.
Then $\Sigma_0 = \left\{\sigma^n\:|\: n\in\ZZ\right\}$ and $\Sigma_1 = \left\{\id,\rho,\ldots,\rho^{m-1}\right\}$.
\end{example}

Let $R$ be a $\Sigma'$-ring and let
$
R\left[\Sigma'\right] = \left\{\sum r_\tau\tau\:\big|\: r_\tau \in R,\ \tau \in \Sigma'\right\}
$
denote the ring of difference operators on $R$. The multiplication on $R[\Sigma']$ is given by
$
\tau\cdot r = \tau(r)\tau$.
For a set $Y$,  let
$$
R\{Y\}_{\Sigma'} = R\left[\ldots,\tau y,\ldots\:|\: \tau \in \Sigma',\ y\in Y\right]
$$
denote the ring of $\Sigma'$-polynomials over $R$ with $Y$ as the set of $\Sigma'$-indeterminates.

\begin{example}
For example, if $\Sigma' = \Sigma_1 = \Z/2\Z$ and $\rho$ is a
generator of $\Sigma_1$, then $R\{y\}_{\Sigma'} = R[y,\rho y]$ with
the action of $\rho$ given by $\rho(y) = \rho y$ and $\rho(\rho y) =
y$.
\end{example}

%\begin{definition}
An ideal $\ia \subset R$ is called a $\Sigma'$-ideal if $\Sigma'(\ia)\subset\ia$, where $$\Sigma'(\ia) := \left\{\sigma(a)\:|\:\sigma\in\Sigma', a\in\ia\right\}.$$
%\end{definition}
The smallest $\Sigma'$-ideal containing a set $F \subset R$ is denoted by $[F]_{\Sigma'}$. If $\Sigma' = \Sigma$, then it is also denoted simply by $[F]$.
%\begin{definition}
 Let $R_1$ and $R_2$ be $\Sigma'$-rings. A ring homomorphism $f : R_1 \to R_2$ is called a $\Sigma'$-homomorphism if
$
f(\tau(r))=\tau(f(r))$, for all $\tau \in \Sigma',$ $r \in R_1$.
%\end{definition}

The following example shows that even if we start with a base field, the constants of the solution space as constructed in Section~\ref{sec:PVTheory}
have zero divisors.
\begin{example}\label{ex:needtouserings}
Let $\Sigma_1=\Z/ 4\Z$ with a
generator $\rho$. Consider the equation
%\begin{equation}\label{eq:sigmax-x}
$\sigma x = - x$.
%\end{equation}
The procedure of constructing a solution space (called Picard--Vessiot extension) of the above equation described in Section~\ref{sec:PVTheory}
first takes $\C\{x,1/x\}_{\rho}$, with $\sigma x = -
x$, and then quotients by  $\left[\rho x-ix, x^4-1\right]$, which is a maximal
$\Sigma$-ideal. Thus, we arrive at the ring
$\C[x]\big/{\left(x^4-1\right)}$, $\sigma x = -x$ and $\rho x = i
x$, which is a $\Sigma$-pseudofield generated by the solution of
the equation. The subring of constants is generated by $x^2$ and is
isomorphic to $\C[t]\big/{\left(t^2-1\right)}$, which is not a field.
\end{example}

Denote the ring of $\Sigma'$-constants of $R$ by $R^{\Sigma'}$. In other words, $$
R^{\Sigma'} = \left\{r \in R\:|\: \tau(r)=r\ \text{for all}\ \tau \in \Sigma'\right\}.$$
The set of all $\Sigma'$-ideals of $R$ will be denoted by
$
\Id^{\Sigma'}(R)$.

\begin{definition} A $\Sigma'$-ideal $\p$ of $R$ is
called pseudoprime if there exists a multiplicatively closed subset
$S \subset R$ such that $\p$ is a maximal $\Sigma$-ideal with
$\p\cap S =\varnothing$.
\end{definition}

\begin{lemma}\label{lem:preimage}Let $A$ and $B$ be $\Sigma$-rings and $\varphi: A\to B$ be a
$\Sigma$-homomorphism. Then for any pseudoprime ideal $\q$ in $B$ the ideal $\varphi^{-1}(\q)$ is pseudoprime.
\end{lemma}
\begin{proof}
See \cite[Section~2]{Dima}.
%Let $S\subset B$ be a multiplicative set such that
%$\q$ is a maximal $\Sigma$-ideal with $\q \cap S = \varnothing$. Then there is
%a prime ideal $\p$ containing $\q$ such that $\p\cap S=\varnothing$.
%Hence, $\varphi^{-1}(\q) \subset A$ is maximal $\Sigma$-ideal with
%$$\varphi^{-1}(\q) \cap A\setminus\varphi^{-1}(\p)= \varnothing.$$ Indeed, let
% $\ia\subset A$ be a $\Sigma$-ideal such that
% $
% \varphi^{-1}(\q) \subset \ia \subset \varphi^{-1}(\p)$. Then,
%$B\varphi(\ia)\subset \p$ is a $\Sigma$-ideal. Therefore, $B\varphi(\ia)\subset \q$. Thus, $\ia\subset \frak \varphi^{-1}(\q)$.
\end{proof}

The set of all pseudoprime ideals of $R$ will be denoted by
$
\PSpec R$ or $\PSpec^{\Sigma'} R$.
For $s \in R$,
$
(\PSpec R)_s$
denotes the set of pseudoprime ideals of $R$ not containing $s$. Let $R_1$ and $R_2$ be $\Sigma'$-rings and $f : R_1\to R_2$ be a $\Sigma'$-homomorphism. Then
$
f^*(\q) := f^{-1}(\q)
$
defines a map $$f^* :\PSpec R_2 \to \PSpec R_1$$ by Lemma~\ref{lem:preimage}. For an ideal $\ia \subset R$ denote by $\ia_{\Sigma'}$ the largest $\Sigma'$-ideal of $R$ contained in $\ia$. Note that if $\p$ is a prime ideal of $R$, then the ideal $\p_{\Sigma'}$ is pseudoprime.

Recall that an $R$-module $M$ with an action of $\Sigma'$ is called a $\Sigma'$-module if for all $\tau \in \Sigma'$, $r \in R$, and $m \in M$, we have $
\tau(rm)=\tau(r)\tau(m)$.
    A $\Sigma'$-ring is called simple if it contains no proper $\Sigma'$-ideals
    except for $(0)$.

\begin{definition} A ring $R$ is called absolutely flat if every $R$-module is flat.
%\end{definition}
%\begin{definition}
An absolutely flat simple $\Sigma'$-ring $\kkk$ is called a  $\Sigma'$-pseudofield (see~\cite{Dima}).
\end{definition}

For every subset $E \subset R\{y_1,\ldots,y_n\}_{\Sigma'}$, let $\V(E) \subset R^n$ be the set of common zeroes of $E$ in $R^n$.
Conversely, for every subset $X \subset R^n$, let $$\I(X) \subset R\{y_1,\ldots,y_n\}_{\Sigma'}$$ be the $\Sigma'$-ideal of all
polynomials in $R\{y_1,\ldots,y_n\}_{\Sigma'}$ vanishing on $X$.  One sees that, for any reduced $R$ and $\Sigma'$-ideal $I \subset R\{y_1,\ldots,y_n\}_{\Sigma'}$, we have
$
\sqrt{I} \subset \I(\V(I))$.
\begin{definition}\label{def:diffclosedring}\cite[Section 4.3]{Dima} A $\Sigma'$-pseudofield $R$ is called {\it difference closed} if, for every $\Sigma'$-ideal $I\subset R\{y_1,\ldots,y_n\}_{\Sigma'}$, we have
$
\sqrt{I} = \I(\V(I))$.
\end{definition}

\subsection{Properties of pseudofields}\label{sec:pseudofields}
%\begin{proposition}\label{prop:diffclosed}\cite[Proposition 25]{Dima}
%Suppose that $|\Sigma'|<\infty$. Then, a $\Sigma'$-pseudofield $U$ is difference closed if and only if, for every finite system
%$
%F=0, G\ne 0$ of $\Sigma'$-equations and inequations, if the system has a solution in some  $\Sigma'$-pseudofield $L\supset U$ then it has a solution in $U$.
%\end{proposition}

%\begin{theorem}\label{thm:diffclosed}\cite[Proposition 19]{Dima} Every $\Sigma'$-pseudofield can be embedded into
%a difference closed pseudofield and there exists a minimal such pseudofield. In particular, every
%$\Sigma'$-field can be embedded into a difference closed $\Sigma'$-pseudofield.
%\end{theorem}

\begin{proposition}\label{prop:1} Let $L$ be $\Sigma'$-simple ring and $K \subset L$
be an absolutely flat $\Sigma'$-subring. Then, $K$ is a $\Sigma'$-pseudofield.
\end{proposition}
\begin{proof} Let $0\neq a \in K$. We will show that the $\Sigma'$-ideal of $K$ generated by $a$ contains $1$. Since $K$ is absolutely flat, we may assume that
%\begin{equation}\label{eq:aidempotent}
$a^2=a$,
%\end{equation}
since every principal ideal is generated by an idempotent~\cite[Exercise~II.27]{AM}.
Since the $\Sigma'$-ideal generated by $a$ in $L$ contains $1$, there exist $h_i \in L$, $0\Le i\Le r$, such that
\begin{equation}\label{eq:1sigmaa}
1 = h_0 a + h_1\sigma_1(a)+\ldots+h_r\sigma_r(a)
\end{equation}
for some $\sigma_k\in\Sigma'$. Set $\sigma_0=\id$ for notation.  We will show by induction on $k\Le r$ that the $h_i$'s can be selected so that $h_i\in K$, $0\Le i \Le k$.
The base $k = 0$ is done in the same way as the inductive step.
Assume the statement for $k-1\geq 0$. We will show it for $k$. Multiplying~\eqref{eq:1sigmaa} by $1-\sigma_k(a)$ and using $a^2=a$, we have:
%\begin{align*}
$$
1 -\sigma_{k}(a)= (1 -\sigma_{k}(a))(h_0\cdot a +\ldots+h_{k-1}\cdot\sigma_{k-1}(a)+h_{k+1}\cdot\sigma_{k+1}(a)+\ldots+h_r\cdot\sigma_r(a)).
$$
%\end{align*}
Hence,
\begin{align*}
1 = &(1 -\sigma_{k}(a))h_0\cdot a +\ldots+(1 -\sigma_{k}(a))h_{k-1}\cdot\sigma_{k-1}(a)+\\
&+\sigma_{k}(a)+(1 -\sigma_{k}(a))h_{k+1}\cdot\sigma_{k+1}(a)+\ldots+(1-\sigma_k(a))h_r\cdot\sigma_r(a)
\end{align*}
with $(1 -\sigma_{k}(a))h_0,\ldots,(1 -\sigma_{k}(a))h_{k-1}, 1 \in K$, which finishes the proof.
 \end{proof}

\begin{proposition}\label{prop:2} Let $L$ be an absolutely flat ring and $H \subset \Aut(L)$. Then the ring $L^H$ is absolutely flat.
\end{proposition}
\begin{proof}
    Let $0\neq a \in L^H$. Then by \cite[Exercise~II.27]{AM} there exist unique an idempotent $e$ and $a'$ in $L$ such that
\begin{equation}\label{eq:eaa}
e = aa',\ \ a=ea,\ \ \text{and}\ \ a'=ea'.
\end{equation}
To see uniqueness, note that
if $(\bar e, \bar a')$
is another such pair, then $e\bar e = e a \bar a'=a \bar a' = \bar e$
and, similarly, $e\bar e =  e$. So, the element $e$ is unique. Now,
$a' = e a' = \bar e a' = a \bar a' a'$ and, in the same manner, $\bar a' = \bar e \bar a' = e \bar a' = a a' \bar a'$.

We will show now that $e$ and $a'$ are $H$-invariant. For $\sigma \in H$ we have
$$
a = \sigma(a) = \sigma(ae) = a\sigma(e).$$
Multiplying by $a'$, we obtain
%\begin{equation}\label{eq:e1}
$e = e\sigma(e)$.
%\end{equation}
Similarly, we obtain
$
e = e\sigma^{-1}(e)$,
which implies that
%\begin{equation}\label{eq:e2}
$\sigma(e)=e\sigma(e)$.
%\end{equation}
Hence,
$
\sigma(e)=e$.
We, therefore, have
\begin{equation}\label{eq:eaasigma}
e = a\sigma(a'),\ \ a=ea,\ \ \text{and}\ \ \sigma(a')=e\sigma(a').
\end{equation}
Since the pair $(e,a')$ is unique,~\eqref{eq:eaa} and~\eqref{eq:eaasigma}
imply that
$
\sigma(a')= a'$.
Applying \cite[Exercise~II.27]{AM} again, we conclude that $L^H$ is absolutely flat.
 \end{proof}

\begin{proposition}\label{prop:26} Let $A$ be a $\Sigma_1$-closed pseudofield. Then the ring $R = A[\Sigma_1]$ is completely reducible:
%\begin{equation}\label{eq:isomRA}
$R \cong A\oplus \ldots \oplus A$
%\end{equation}
as $\Sigma_1$-modules over $A$. In other words, every $\Sigma_1$-module over $A$ has a basis  of $\Sigma_1$-invariant elements.
Moreover,
$
A[\Sigma_1] \cong \Mn_n(C)$
as rings, where $C = A^{\Sigma_1}$.
\end{proposition}
\begin{proof}
Follows \cite[Proposition~26 and Remark~27]{Dima}.
%, we only need to show that every $\Sigma_1$-module over $A$ has a basis  of $\Sigma_1$-invariant elements.
%For this, first recall that  every left module of a ring $R$ is a direct sum of irreducible submodules if and only if the ring $R$ is a direct sum of irreducible left
%ideals \cite[Theorem~4.3, Chapter~XVII]{Lang}. Moreover, if the ring $R$ has decomposition $R\cong V_1\oplus
%\ldots\oplus V_n$ then every $R$-module is a direct sum of
%submodules each isomorphic to some of the $V_i$'s  \cite[Theorem~4.4, Chapter~XVII]{Lang}. Every $\Sigma_1$-module over a
%$\Sigma_1$-ring $A$ is an $A[\Sigma_1]$-module.
%Each summand in~\eqref{eq:isomRA} has a $\Sigma_1$-invariant $A$-basis consisting of just $1$. Combining this with the above isomorphisms, we have the desired result.
 \end{proof}

\begin{proposition}\label{prop:idealcorrespondance} Let $R$ be a $\Sigma$-simple ring and $A := R^\sigma$ be a $\Sigma_1$-difference closed pseudofield. Let $B$ be any $\Sigma$-$A$-algebra with $\sigma$ acting as the identity. Then the $\Sigma$-homomorphism
$
B \to R\otimes_A B$, with $b\mapsto 1\otimes b,$ $b\in B$,
induces a bijection
$$
\Id^{\Sigma_1}(B) \longleftrightarrow \Id^\Sigma(R\otimes_AB)
$$
via
$
\ia \subset B \longrightarrow \ia^e := R\otimes_A\ia$,
$\ib^c := \ib\cap B \longleftarrow \ib \subset R\otimes_AB$.
\end{proposition}
\begin{proof}
Let $I$ be a $\Sigma$-ideal of the ring $R\otimes_AB$ and let $I^c=J$. We will show that $I = J^e$. In other words,
by passing to $R\otimes_A(B/J)$, we will show that if $I^c = (0)$,
then $I = (0)$. By Proposition~\ref{prop:26}, there exists a basis $\{b_i\}_{i\in \mathcal{I}}$ of $B$ over $A$ consisting of $\Sigma_1$-invariant elements. Then, every element of $R\otimes_A B$ is of the form
$$
a_1\otimes b_{i_1}+\ldots+a_n\otimes b_{i_n}
$$
for some $a_i \in R$, $1\Le i\Le n$. Let $0\ne u \in I$ have the shortest
expression of the form
$
u = a_1\otimes b_{j_1}+\ldots+a_k\otimes b_{j_k}$,
$$
M = \left\{a \in R\:|\: \exists c_2,\ldots,c_k \in R, i_1,\ldots,i_k\in\mathcal{I}: a\otimes b_{i_1}+c_2\otimes b_{i_2}+\ldots+c_k\otimes b_{i_k} \in I\right\}.
$$
As $0\neq a_1\in M$, and
since $\Sigma(b_i) = b_i$, $1\Le i\Le n$, the set $M$ is a non-zero $\Sigma$-ideal of $R$. Hence, $1 \in M$. Therefore, there exists $u$ with $a_1 = 1$. Since
\begin{equation}\label{eq:usigmau1}
u - \sigma(u) = (a_2 -\sigma(a_2))\otimes b_{i_2}+\ldots+  (a_k -\sigma(a_k))\otimes b_{i_k} \in I
\end{equation}
and has a shorter expression than $u$, we have
\begin{equation}\label{eq:usigmau2}
u - \sigma(u)=0.
\end{equation} Since $\{b_i\}_{i\in\mathcal{I}}$ is a basis of $B$ over $A$,
$
\{1\otimes b_i\}_{i\in\mathcal{I}}
$
is a basis of $R\otimes_A B$ over $R$. Therefore,~\eqref{eq:usigmau1} and~\eqref{eq:usigmau2} imply that
$
\sigma(a_2)=a_2,\ldots,\sigma(a_k)=a_k,$
that is, $a_2,\ldots,a_k \in A$. Thus,
$$
u= 1\otimes\left(b_{i_1}+a_2b_{i_2}+\ldots+a_k b_{i_k}\right).$$
Hence,
$$
0\ne b_{i_1}+a_2b_{i_2}+\ldots+a_k b_{i_k} \in I^c,$$
contradicting $I^c = (0)$. Therefore, we have shown that
$
{(I^c)}^e = I$.
On the other hand, since $R$ is a free $A$-module, the $B$-module $R\otimes_A B$ is also free and, therefore, faithfully flat. Thus, by \cite[Exercise~III.16]{AM} for every ideal  $J \subset B$ we have
$
{(J^e)}^c = J$,
which finishes the proof.
 \end{proof}

\begin{corollary}\label{cor:tensor} Let $B$ be a $\Sigma$-ring containing a $\Sigma$-pseudofield $L$ with $C_L:= L^\sigma$ being a $\Sigma_1$-closed pseudofield. Let $C \subset B^\sigma $ be a
    $\Sigma_1$-subring such that $C_L \subset C$. Then
$
L\cdot C = L\otimes_{C_L} C$.
\end{corollary}
\begin{proof}
The kernel $I$ of the $\Sigma$-homomorphism $$
L\otimes_{C_L} C \to L\cdot C\subset B,\quad l\otimes c\mapsto l\cdot c,$$
is a $\Sigma$-ideal with $I^c = (0) \subset C$. By Proposition~\ref{prop:idealcorrespondance}, we conclude that $I = 0$.
 \end{proof}

\subsection{Noetherian pseudofields}\label{sec:pseudofieldsN}
\begin{lemma}\label{lem:specsurjective} Let $A\subset B$ be $\Sigma$-rings such that for some $s \in A$ the map $\Spec B_s \to \Spec A_s$ is surjective. Then the map
$
\varphi : (\PSpec B)_s \to (\PSpec A)_s$
is surjective as well.
\end{lemma}
\begin{proof}
    Let $\q \subset A$ be a pseudoprime ideal with $s \notin \q$. Then, since the maximal ideal not intersecting a multiplicative subset is prime, by definition, there exists a prime ideal $\p \supset \q$ such that
$$
\q = \bigcap\nolimits_{\tau \in \Sigma}\tau(\p)$$
with $\q$ being a maximal $\Sigma$-ideal contained in $\tau(\p)$, $\tau\in\Sigma$. Since $s \notin\q$, there exists $\tau \in \Sigma$ such that $s \notin \tau(\p)$. By our assumption, there exists a prime ideal $\p' \subset B$ with $\p'\cap A = \p^\tau$. Then the ideal $\p'_\Sigma$ is the pseudoprime ideal in $B$ that is mapped to $\p$ by $\varphi$.
 \end{proof}

\begin{lemma}\label{lem:s} Let $A \subset B$ be $\Sigma$-rings such that $A$ is Noetherian and reduced and $B$ is a finitely generated $A$-algebra. Then there exists $0\ne s \in A$ such that the map
$
(\PSpec B)_s \to (\PSpec A)_s$
is surjective.
\end{lemma}
\begin{proof}
    There exists $s\in A$ such that $A_s$ is an integral domain. For instance, suppose that $(0)=\mathfrak{p}_1\cap\cdots\cap\mathfrak{p}_t$ is the representation of $(0)$ as
    the intersection of the finitely many minimal prime ideals in the Noetherian ring $A$. Let $s\in\mathfrak{p}_2\cap\cdots\cap\mathfrak{p}_t$ be such that $t\notin\mathfrak{p}_1$.
    Then, $A_s$ is a reduced ring with a single minimal prime ideal. Thus, it is integral.
    By \cite[Lemma~30]{Dima}, there exists $t \in A$ such that the map
$
\Spec B_{st} \to \Spec A_{st}
$
is surjective. The statement now follows from Lemma~\ref{lem:specsurjective}.
 \end{proof}

\begin{theorem}\label{thm:nonewconstants} Let $L$ be a Noetherian $\Sigma$-pseudofield with $C := L^\sigma$ being a $\Sigma_1$-closed pseudofield.
    Let $R$ be a $\Sigma_1$-finitely generated $\Sigma$-simple ring over $L$. Then
$
R^\sigma = C$.
\end{theorem}
    \begin{proof}
    Let $b \in R^\sigma$. Since $|\Sigma_1|< \infty$, the ring $R$ is finitely generated
    over $L$. Since $R$ is $\Sigma$-simple, it is reduced. Therefore, the $\Sigma$-subring of $R$ generated by $L$ and $b$, denoted by $L\{b\}$,  is reduced as well.
    Hence, by Lemma~\ref{lem:s},
    there exists a non-nilpotent element $s \in L\{b\}$ such that the map
    $$
    (\PSpec R)_s \to (\PSpec L\{b\})_s
    $$
    is surjective. Therefore, since $\PSpec R = \{(0)\}$, every non-zero pseudoprime ideal in $L\{b\}$ contains $s$. By Corollary~\ref{cor:tensor}, we have $
    L\{b\} = L\otimes_C C\{b\}$.
    By Proposition~\ref{prop:26}, $L$ is a free $C$-module. Let $\{l_i\}_{i\in\mathcal{I}}$ be a $\Sigma_1$-invariant basis over $C$. Then there exist $r_1,\ldots, r_k \in C\{b\}$ such that
    $$
    s = l_1\otimes r_1+\ldots +l_k\otimes r_k.$$
    Since the ring $L\{b\}$ is reduced, $r_1$ is not nilpotent. Therefore, by \cite[Proposition~34]{Dima},
    there exists a maximal $\Sigma$-ideal $\m$ in $C\{b\}$ such that $C\{b\}/\m = C$ and  $r_1 \notin \m$.
    Let
    $$
    \varphi : L\{b\} = L\otimes_CC\{b\}\to L\otimes_CC\{b\}/\m = L\otimes _CC = L.
    $$
    Then,
    $$
    \varphi(s) = l_1\bar{r}_1+\ldots+l_k\bar{r}_k,$$
    where $\bar{r}_i$ are the images of $r_i$ modulo $\m$, $1\Le i\Le k$. Since $\{l_1,\ldots,l_k\}$ are linearly independent over $C$ and $\bar{r}_1\ne 0$,
    the ideal $L\otimes_C\m$ does not contain $s$. Since $\varphi$ is a $\Sigma$-homomorphism, $L\otimes_C\m = \varphi^{-1}((0))$, and $(0)$ is a pseudoprime ideal in $L$,
    the ideal $L\otimes_C \m$ is pseudoprime by Lemma~\ref{lem:preimage}. Therefore, $L\otimes_C \m = (0)$ by the above. Thus, we see that $b \in C$ by taking $\sigma$-invariants as $\varphi$ is an injective $\Sigma$-homomorphism.
 \end{proof}

%\begin{definition}
Recall that an idempotent that is not a sum of several distinct orthogonal idempotents is called  indecomposable.
%\end{definition}

\begin{proposition} Let $L$ be a Noetherian $\Sigma$-pseudofield and let $F = L/\m$, where $\m$ is a maximal ideal in $L$. Then,
$
L \iso F\times\ldots\times F$.
Moreover, $\Sigma$ acts transitively on the set of indecomposable idempotents of $L$.
\end{proposition}
\begin{proof}
Since the ring $L$ is Noetherian and $\dim L = 0$, by \cite[Theorem~8.5]{AM}, the ring $L$ is Artinian. Therefore, by \cite[Theorem~VII.7]{AM}, it is a finite product of local Artinian rings.
Since $L$ is reduced, by \cite[Proposition~VIII.1]{AM},
\begin{equation}\label{eq:decomposition}
L = F_1\times\ldots\times F_n,
\end{equation}
where $F_i$ is a field, $1\Le i\Le n$. Since $L$ is $\Sigma$-simple, the group $\Sigma$ acts transitively on $\Spec L$.
Therefore, $F_i \cong F_1$, $1\Le i\Le n$, as residue fields. Let $e$ be an indecomposable idempotent in $L$. Let $\Orb_\Sigma(e) = \{e_1,\ldots,e_k\}$. Then, the idempotent
$$
E := e_1+\ldots+e_k$$
is $\Sigma$-invariant. Since $L$ is $\Sigma$-simple, we have $E = 1$. Decomposition~\eqref{eq:decomposition} implies that $L$ has $n$ indecomposable idempotents, each one is of the form
$$
(0,\ldots,0,1,0,\ldots,0)$$
and, thus, $k = n$ and $\Sigma$ acts transitively on the set of indecomposable idempotents of $L$.
 \end{proof}

Let $B$ be a $\Sigma_0$-ring and let
\begin{equation}\label{eq:defofFSigma}
F_{\Sigma_1}(B) = \prod_{\mu \in \Sigma_1}B = \{f : \Sigma_1\to B\},
\end{equation}
which is a $\Sigma_0$-ring with the component-wise action of $\Sigma_0$. Define
$$
(\mu f) (\tau) = f\left(\mu^{-1}\tau\right),\quad f\in F_{\Sigma_1}(B),\ \ \mu,\tau\in\Sigma_1.$$
The above makes $F_{\Sigma_1}(B)$ a $\Sigma$-ring. For every $\mu\in\Sigma_1$ define a $\Sigma_0$-homomorphism
\begin{equation}\label{eq:gammamu}
\gamma_\mu : F_{\Sigma_1}(B) \to B,\quad f\mapsto f(\mu).
\end{equation}
Moreover, we have $$\gamma_\tau(\mu f) = (\mu f)(\tau) = f\left(\mu^{-1}\tau\right) = \gamma_{\mu^{-1}\tau}(f).$$

\begin{proposition}\label{prop:10} Let $A$ be a $\Sigma$-ring, $B$ be a $\Sigma_0$-ring, and $\varphi : A \to B$ be a $\Sigma_0$-homomorphism. Then for every $\mu \in \Sigma$ there exists unique $\Sigma$-homomorphism $\Phi_\mu : A \to F_{\Sigma_1}(B)$ such that the following diagram
$$
\xymatrix{
                                    & {F_{\Sigma_1}(B)}\ar[d]^{\gamma_{\mu}} \\
        A\ar[r]^{\varphi}\ar[ur]^{\Phi_\mu} & B
}
$$
is commutative.
\end{proposition}
\begin{proof}
Since $$
\Phi_\mu(a)\left(\tau^{-1}\mu\right)=(\tau\Phi_\mu(a))(\mu)=\varphi(\tau a),$$
where $a\in A$ and $\tau\in\Sigma$, the homomorphism $\Phi_\mu$ is unique if it exists. Define
$$
\Phi_\mu(a)(\tau) = \varphi\left(\mu\tau^{-1}a\right).$$
For every $\alpha \in \Sigma_1$ we have
\begin{align*}
\Phi_\mu(\alpha a)(\tau) &= \varphi\big(\mu\tau^{-1}\alpha a\big) = \varphi\big(\mu\big(\alpha^{-1}\tau\big)^{-1} a\big) =\Phi_\mu(a)\left(\alpha^{-1}\tau\right)
=(\alpha\Phi_\mu(a))(\tau)\\
\Phi_\mu(\nu a)(\tau)&= \varphi\left(\mu\tau^{-1}\nu a\right) = \nu\left(\varphi\left(\mu\tau^{-1}a\right)\right)=\nu(\Phi_\mu(a)(\tau))=\nu(\Phi_\mu(a))(\tau)
\end{align*}
for all $\alpha,\tau \in \Sigma_1,$ $\nu \in \Sigma_0$, and $a \in A$. Thus, $\Phi_\mu$ is a $\Sigma$-homomorphism.
 \end{proof}

\begin{proposition}\label{prop:11} Let $L$ be a Noetherian $\Sigma$-pseudofield such that $L^\sigma$ is a $\Sigma_1$-closed pseudofield.
    Then, there exists a Noetherian $\Sigma_0$-pseudofield $B$ such that $L \iso F_{\Sigma_1}(B)$.
\end{proposition}
\begin{proof}
By \cite[Theorem~17(4)]{Dima}, there exists an algebraically closed field $K$ such that
$$
L^\sigma = F_{\Sigma_1}(K).$$ Let $\delta_\tau \in F_{\Sigma_1}(K)$ be the indicator of the point $\tau \in \Sigma_1$ and
$e = \delta_{\id}$.
%\quad \text{by}\quad e(\tau) = \begin{cases}
%1,&\tau = \id, \\
%0,&\tau \ne \id.
%\end{cases}$$
Let also
$
B = eL$,
which is a Noetherian absolutely flat ring as a quotient of a Noetherian $\Sigma$-pseudofield.
By
Proposition~\ref{prop:10}, the homomorphism $
L \to B,$ with $a \mapsto e\cdot a$,
lifts to a unique $\Sigma$-homomorphism $$\phi: L \to F_{\Sigma_1}(B).$$
Since $L$ is $\Sigma$-simple,  $\phi$ is injective. To show that $\phi$ is surjective, we will prove that $\phi(L)$ contains all indecomposable idempotents of $F_{\Sigma_1}(B)$. Every indecomposable idempotent of the ring $F_{\Sigma_1}(B)$ is of the form
$\delta_\tau\cdot f$,
%where $\delta_\tau$ is the indicator of the point $\tau$
%and
$f$ is an indecomposable idempotent of $B$. Let $f = eh$, where $h\in L$. Since $$\phi(\tau(e)h)(\nu) = (e\tau(e)h)(\nu)=(\tau(e)f)(\nu) = e\left(\tau^{-1}\nu\right)f=\delta_\tau(\nu) f,$$ we are done.
Finally, $B$ is $\Sigma_0$-simple. Indeed,  let $\ib \subset B$ be
a $\Sigma_0$-ideal. Let $I \subset F_{\Sigma_1}(B)$
consist of all functions $f$ with image contained in $\ib$.
Since $I$ is an ideal and $\Sigma_1$ is acting on the
domain, $I$ is invariant under the $\Sigma_1$-action.
Since $\ib$ is a $\Sigma_0$ ideal, then $I$  is a $\Sigma_0$-ideal as well. Therefore, $I$ is a $\Sigma$-ideal, which contradicts to $L$ being a
pseudofield.
 \end{proof}

\begin{proposition}\label{prop:product} Let $L$ be a Noetherian $\Sigma$-pseudofield such that $L^\sigma$ is a $\Sigma_1$-closed pseudofield. Then,
$$
L \cong \prod_{i=1}^n F_{\Sigma_1}(F)
$$
as $\Sigma_1$-rings, where $F$ is a field.
\end{proposition}
\begin{proof}
By Proposition~\ref{prop:11},
$
L = F_{\Sigma_1}(B)$,
where $B$ is a Noetherian $\Sigma_0$-pseudofield. Let $f_1,\ldots, f_n$ be all
indecomposable idempotents of $B$. Then $$
L = f_1L\times\ldots\times f_nL.$$
On the other hand,
$$
f_iF_{\Sigma_1}(B) = F_{\Sigma_1}(f_iB)=F_{\Sigma_1}(F_i),$$
where $F_i = f_i B$ and $F_1 \cong F_i$, $1\Le i \Le n$.
 \end{proof}

\begin{proposition}\label{prop:Noetherian} Let $L$ be a Noetherian $\Sigma$-pseudofield and $K \subset L$ be a $\Sigma$-pseudofield as well. Then $K$ is Noetherian.
\end{proposition}
\begin{proof} Note that a pseudofield is Noetherian if and only if it contains a finite set of indecomposable idempotents $e_1,\ldots, e_n$ with
\begin{equation}\label{eq:indecomposables}
e_1+\ldots+e_n = 1.
\end{equation}
Necessity has been discussed above. To show sufficiency, note that if $e$ is an indecomposable idempotent of an absolutely flat ring $R$, then $eR$ is a field. Indeed, $eR$ is an
absolutely flat ring without nontrivial idempotents \cite[Exercise~II.27]{AM}. Moreover, for every
element $x\in R$ we have $x=ax^2$. Therefore, $ax$ is an idempotent. So,
either $ax=0$ and, thus, $x=ax^2=0$, or $ax=1$.
Hence, equality~\eqref{eq:indecomposables} implies that $R$ is finite product of fields and, therefore, is Noetherian.

Thus, since every idempotent of $K$ is an idempotent of $L$, which is Noetherian, the ring $K$ has finitely many indecomposable idempotents $f_1,\ldots, f_k$.
Since $f_1+\ldots+f_k$ is left fixed by $\Sigma$, we have
$
f_1+\ldots+f_k =1$. Again, by the above, the ring $K$ is Noetherian.
 \end{proof}

\begin{proposition} Let $L$ be a $\Sigma$-field such that the subfield $C := L^\sigma$ is algebraically closed. Then there exists a $\Sigma$-pseudofield $A$ and a $\Sigma$-embedding $\varphi : L \to A$ such that $A^\sigma$
is the $\Sigma_1$-closure of the $\Sigma_1$-field $\varphi(C)$.
\end{proposition}
\begin{proof} Set $A = F_{\Sigma_1}(L)$ and
and let $\varphi$ be the Taylor homomorphism for $\id : L\to L$ by Proposition~\ref{prop:10}.
%
% let $\varphi(l)_\mu := \mu^{-1}(l)$. We have:
%$$
%\tau(\varphi(l))_\mu=\varphi(l)_{\tau^{-1}\mu}=(\tau^{-1}\mu)^{-1}(l)=(\mu^{-1}\tau) (l) = \varphi(\tau l)_\mu,
%$$ where $l \in L$ and $\tau$, $\mu\in\Sigma_1$.
Then, $A^\sigma = F_{\Sigma_1}(C)$ is the $\Sigma_1$-closure
of $C$ \cite[discussions preceding Proposition~19]{Dima}.
 \end{proof}

\section{Picard--Vessiot theory}\label{sec:PVTheory}
\subsection{Picard--Vessiot ring}
 Let $K$ be a Noetherian
$\Sigma$-pseudofield and let $C = K^\sigma$ be a $\Sigma_1$-closed
pseudofield. Let $A \in \GL_n(K)$. Consider the following difference
equation
\begin{equation}\label{eq:diffeq}
\sigma Y = AY.
\end{equation}
Let $R$ be a $\Sigma$-ring containing $K$.
\begin{definition} A matrix $F \in \GL_n(R)$ is called a fundamental matrix of equation~\eqref{eq:diffeq} if $\sigma F = AF$.
\end{definition}
Let $F_1$ and $F_2$ be two fundamental matrices of~\eqref{eq:diffeq}. Then for $M := F_1^{-1}F_2$ we have
$$
\sigma(M) = \sigma(F_1)^{-1}\sigma(F_2) = F_1^{-1}A^{-1}AF_2 = F_1^{-1}F_2 = M,$$
that is, $M \in \GL_n\left(R^\sigma\right)$.
\begin{definition}
A $\Sigma$-ring $R$ is called a Picard--Vessiot ring for
equation~\eqref{eq:diffeq} if
\begin{enumerate}
\item there exists a fundamental matrix $F \in \GL_n(R)$ for~\eqref{eq:diffeq},
\item $R$ is a $\Sigma$-simple ring, and
\item $R$ is $\Sigma$-generated over $K$ be the matrix entries $F_{ij}$ and $1/\det F$.
\end{enumerate}
\end{definition}

\begin{proposition}\label{prop:nonewconst}
Let $K$ be a Noetherian $\Sigma$-pseudofield, $K^\sigma$ be a
$\Sigma_1$-closed pseudofield, and $R$ be a Picard--Vessiot ring for
equation~\eqref{eq:diffeq}. Then, $
R^\sigma = K^\sigma$.
\end{proposition}
\begin{proof} Since $R$ is a $\Sigma_1$-finitely generated algebra over $K$ and
    $|\Sigma_1|<\infty$, $R$ is finitely generated over $K$. Then the result follows from Theorem~\ref{thm:nonewconstants}.
 \end{proof}

\begin{proposition}\label{prop:PVringunique} Let $K$ be a Noetherian $\Sigma$-pseudofield with $K^\sigma$ being a $\Sigma_1$-closed pseudofield. Then there exists a unique Picard--Vessiot ring for equation~\eqref{eq:diffeq}.
\end{proposition}
\begin{proof} For existence, define the action of $\sigma$ on the $\Sigma_1$-ring $$R:=K\{F_{ij},1/\det F\}_{\Sigma_1}$$ by $\sigma F = AF$. Let $\m$ be any maximal $\Sigma$-ideal in $R$. Then $R/\m$ is the Picard--Vessiot ring for equation~\eqref{eq:diffeq}.
For uniqueness, let $R_1$ and $R_2$ be two
Picard--Vessiot rings of equations~\eqref{eq:diffeq}. Let
$$R=(R_1\otimes_K R_2)/\m,$$ where $\m$ is a maximal $\Sigma$-ideal. Since $R_1$ and $R_2$ are
$\Sigma$-simple, the $\Sigma$-homomorphisms
$$
\varphi_1 : R_1 \to R, \ r\mapsto r\otimes 1,\quad \varphi_2 : R_2 \to R, \ r\mapsto 1\otimes r,
$$
are injective. Let $F_1$ and $F_2$ be fundamental matrices of $R_1$ and $R_2$, respectively. Then there exists $M \in \GL_n(R^\sigma)$ such that
$
\varphi_1(F_1) = \varphi_2(F_2) M$.
Proposition~\ref{prop:nonewconst} implies that $R^\sigma=K^\sigma$. Therefore, $\varphi_1(F_1) \subset \varphi_2(R_2)$. Similarly, $\varphi_2(F_2) \subset \varphi_1(R_1)$. Hence, $\varphi_1(R_1) = \varphi_2(R_2)$
and, thus, $R_1 \cong R_2 \cong R$.
 \end{proof}
\begin{proposition}\label{prop:isnoetherian} Let $K$ be a Noetherian $\Sigma$-pseudofield with $K^\sigma$ being a $\Sigma_1$-closed pseudofield and $R$ be a Picard--Vessiot ring of equation~\eqref{eq:diffeq}. Then the complete quotient ring $L := \Qt(R)$ is a Noetherian $\Sigma$-pseudofield with $L^\sigma=K^\sigma$.
\end{proposition}
\begin{proof} We will first show that $L$ is $\Sigma$-simple. Let $\ia$ be a non-zero $\Sigma$-ideal of $L$. Then $\ia\cap R \ne (0)$ and, therefore, $1 \in \ia$.

We will now show that $L$ is a finite product of fields. Since the ring $K$ is Noetherian and $R$ is finitely generated over $K$, the ring $R$ is Noetherian as well by the Hilbert basis theorem.
Hence, there exists a smallest set of prime ideals $\p_1,\ldots, \p_n$ in $R$ such that
$
(0) = \p_1\cap\ldots\cap\p_n$.
The set of non-zero divisors in $R$ coincides with $R \setminus \bigcup_{i=1}^n\p_i$.
In $\Qt(R)$, all prime ideals correspond to the $\p_i$'s, that is, they are all  maximal and their intersection is $(0)$. Therefore, by \cite[Proposition~1.10]{AM}
$$
\Qt(R) \cong \Qt(R/\p_1)\times\ldots\times\Qt(R/\p_n),
$$
which is absolutely flat and Noetherian.

Let $c = \frac{a}{b} \in L^\sigma$. Using
Theorem~\ref{thm:nonewconstants}, it suffices to show that
$R\{c\}_{\Sigma_1}$ is a $\Sigma$-simple $\Sigma$-ring, since this
would imply that $c\in K^\sigma$.
For this, we will show that
every $\Sigma$-subring $D \subset L$ containing $K$ is
$\Sigma$-simple. Indeed, for every $0\ne d\in L$ there exists
$a\in R$ such that $0\ne ad\in R$, which is true because $L$ is
the localization with respect to the set of non-zero divisors. Therefore, for every nonzero ideal $\ia$ of $D$ we have $\ia\cap R \neq \{0\}$.
Since $R$ is $\Sigma$-simple, $1 \in \ia$.  \end{proof}

\subsection{Picard--Vessiot pseudofield}
Let $K$ be a Noetherian $\Sigma$-pseudofield with $K^\sigma$ being $\Sigma_1$-closed.

\begin{definition} A Noetherian $\Sigma$-pseudofield $L$ is called a Picard--Vessiot pseudofield for equation~\eqref{eq:diffeq} if
\begin{enumerate}
\item there is a fundamental matrix $F$ of equation~\eqref{eq:diffeq} with coefficients in $L$,
\item $L^\sigma = K^\sigma$,
\item $L$ is $\Sigma_1$-generated over $K$ by the entries of $F$.
\end{enumerate}
\end{definition}
It follows from Proposition~\ref{prop:isnoetherian} that every equation~\eqref{eq:diffeq} has a Picard--Vessiot pseudofield. We will show that all Picard--Vessiot pseudofields are of this form.
\begin{proposition}\label{prop:LQtR} Let $K$ be a Noetherian $\Sigma$-pseudofield, with $C := K^\sigma$ being a $\Sigma_1$-closed pseudofield, and $L$ be a Picard--Vessiot pseudofield for equation~\eqref{eq:diffeq}. Then, $
L \iso \Qt(R)$,
where $R$ is the corresponding Picard--Vessiot ring.
\end{proposition}
\begin{proof}
    Let $\sigma$ act on the $\Sigma_1$-ring $R := L\{X_{ij},1/\det X\}_{\Sigma_1}$ by $\sigma X = AX$. Let $F$ be a fundamental matrix of~\eqref{eq:diffeq} with coefficients in $L$. Define $
Y = F^{-1}X$.
Then $R=L\{Y_{ij},1/\det Y\}_{\Sigma_1}$ and $\sigma Y = Y$. Therefore,
$$
R^\sigma = C\{Y_{ij},1/\det Y\}_{\Sigma_1}.$$
Moreover, we have a $\Sigma$-isomorphism
\begin{equation}\label{eq:LKLCisomorphism}
L\otimes_KK\{X_{ij},1/\det X\}_{\Sigma_1} \cong L\otimes_CC\{Y_{ij},1/\det Y\}_{\Sigma_1}.
\end{equation}
Recall that the Picard--Vessiot ring is given by $
R = K\{X_{ij},1/\det X\}_{\Sigma_1}/I$,
where $I$ is a maximal $\Sigma$-ideal. By Proposition~\ref{prop:idealcorrespondance}
and isomorphism~\eqref{eq:LKLCisomorphism},
the ideal $L\otimes_K I$
corresponds to a $\Sigma$-ideal of the form $L\otimes_CJ$, where $J$ is a $\Sigma_1$-ideal of $C\{Y_{ij},1/\det Y\}_{\Sigma_1}$. This induces a $\Sigma$-isomorphism
$$
\phi : L\otimes_{K}R \to L\otimes_CB,$$
where $B = C\{Y_{ij},1/\det C\}_{\Sigma_1}/J$ consists of $\sigma$-constants. Let $\m$ be a maximal $\Sigma$-ideal in $B$. By \cite[Proposition~14]{Dima},
we have $$
\gamma : B \to B/\m\cong C,$$
since $C$ is a $\Sigma_1$-closed pseudofield.
Let $\varphi$ be the $\Sigma$-homomorphism defined by
$$
\begin{CD}
R@>r\mapsto 1\otimes r>>L\otimes_KR@>\phi>>L\otimes_C B@>\id_L\otimes\gamma>>L\otimes_CC@>l\otimes c\mapsto l\cdot c>>L.
\end{CD}
$$
Since $R$ is $\Sigma$-simple, the homomorphism $\varphi$ is injective. By the universal property, $\varphi$ extends to a $\Sigma$-embedding $\overline\varphi$ of $\Qt(R)$ into $L$. Since $L$ is generated by the entries of its fundamental matrix $F$, we finally conclude that $\overline{\varphi}(\Qt(R)) = L$.
 \end{proof}

\subsection{Difference algebraic groups}\label{sec:diffalggroups}
\subsubsection{Definitions}
In analogy with differential algebraic groups~\cite{Cassidy,CassidyRep}, we make the following definitions. Throughout, $C$ will denote a $\Sigma_1$-closed pseudofield. 
% Let $C\{y_1,\ldots,y_n\}$ be the ring of difference polynomials
%    in the difference indeterminates $y_1,\ldots, y_n$.
 %Recall that the category of $C$-$\Sigma_1$-algebras $\mathcal{A}_{C,\Sigma_1}$
 %   has as morphisms the $C$-algebra maps that commute with $\Sigma_1$.
Recall that for $E\subseteq
    C\{y_1,\ldots,y_n\}_{\Sigma_1}$, the set $\V(E)$ is  the set of
    all common zeroes for elements of $E$ in $C^n$. The sets of the form
    $\V(E)$, for some $E$, are called $C$-$\Sigma_1$-algebraic
    varieties (also called pseudovarieties in~\cite{Dima}).
    
    Also recall that, for an arbitrary $C$-$\Sigma_1$-algebraic variety $X$, the set of
    all difference polynomials vanishing on $X$ will be denoted by
    $\I(X)$. Every difference polynomial defines a polynomial
    function on $X$. The ring of all polynomial functions, thus,
    coincides with $C\{y_1,\ldots,y_n\}_{\Sigma_1}/\I(X)$.

\begin{definition}
    A regular map $f\colon X\to Y$ of $C$-$\Sigma_1$-algebraic varieties is a map given by
    difference polynomials in coordinates (for a general definition see~\cite[Section~4.6]{Dima}, in particular, Theorem~40 there).
\end{definition}

\begin{definition}
    A $C$-$\Sigma_1$-algebraic group is a group supplied with a structure of a
    $C$-$\Sigma_1$-algebraic variety such that the multiplication and
    inverse maps are regular.
\end{definition}

\begin{definition}
    A $C$-$\Sigma_1$-Hopf algebra is a $C$-$\Sigma_1$-algebra $H$ supplied with comultiplication, counit, and antipode morphisms that are all $C$-$\Sigma_1$-algebra morphisms.
\end{definition}

Note that the ring of polynomial functions of a
$C$-$\Sigma_1$-algebraic variety is a reduced Hopf algebra such that
the comultiplication, antipod, and counit are  homomorphisms
of $C$-$\Sigma_1$-algebras.

%\begin{definition}
 %   An affine $C$-$\Sigma_1$-algebraic group $G$ is a functor $G:\mathcal{A}_{C,\Sigma_1}\rightarrow\mathrm{Groups}$ defined by
 %   \begin{equation*}
 %       G(R)=\Hom_{\Sigma_1}(H,R),
  %  \end{equation*}
 %   where $H$ is a $C$-$\Sigma_1$-Hopf-algebra.
%\end{definition}
%\begin{definition}
An example of a $C$-$\Sigma_1$-algebraic group is the group
$\GL_{m,\Sigma_1}(C)$, that is the set of all $m\times m$ matrices
with coefficients in $C$ and having invertible determinant. The
corresponding ring of regular functions is
   % \begin{equation}\label{eq:gln}
        $$H_m=C\{x_{11},\ldots,x_{mm},1/{\det X}\}_{\Sigma_1}.$$
 %   \end{equation}
    The $C$-$\Sigma_1$-algebra $H_m$ has a Hopf algebra structure defined on the $\Sigma_1$-generators in the usual way and is extended by commuting to the $\Sigma_1$-monomials in the generators.
    %Then, we let $\GL_{m,\Sigma_1}(C)$ be the affine $C$-$\Sigma_1$-algebraic group corepresented by $H_m$ as above.
  %  \end{definition}

\begin{example}
   Let us describe the structure of $\GL_{1,\Sigma_1}(C)$ explicitly. Let $$\Sigma_1 = \left\{\id,\rho,\rho^2,\ldots,\rho^{t-1}\right\}$$ and consider
    \begin{equation*}
        H_1=C\{x,1/x\}_{\Sigma_1}=C\left[x,1/x,\rho(x),1/\rho(x),\ldots,\rho^{t-1}(x),1/\rho^{t-1}(x)\right].
    \end{equation*}
    Then, the comultiplication is
    $$\rho^l(x)\mapsto \rho^l(x)\otimes\rho^l(x),$$ and the antipode map is $$\rho^l(x)\mapsto1/\rho^l(x).$$
    Since $C$ is $\Sigma_1$-closed,  it is of the form $F_{\Sigma_1}(K)$
    for some algebraically closed field $K$. Then the group
    $\GL_{1,\Sigma_1}(C)$ has a natural structure of a $K$-algebraic group such
    that $$\GL_{1,\Sigma_1}(C)=\GL_{1}(K)^t.$$
\end{example}

\begin{definition}
    A linear $C$-$\Sigma_1$-algebraic group is a closed subgroup in $\GL_{m,\Sigma_1}(C)$, that is a subgroup given by difference
    polynomials.
\end{definition}
In particular, this means that the $C$-$\Sigma_1$-Hopf algebra $H$
of a linear $C$-$\Sigma_1$-algebraic group is a quotient of $H_m$ by a
radical $\Sigma_1$-Hopf-ideal.
%
%the Yoneda lemma
 %   (see \cite[Corollary 2, page 44]{Pareigis}.
    %or \cite[Corollary 30.7, page 224]{Herrlich}).
    More explicitly, the above equivalence also follows from the equivalence of the
    categories of affine pseudovarieties and the category of reduced $\Sigma_1$-finitely generated algebras~\cite[Proposition~42]{Dima}.

\subsubsection{Difference algebraic subgroups of ${\Gm}_{,\Sigma_1}$}
\begin{example}\label{ex:subgroupsofGm}
    In the usual case of varieties over a field $\kkk$, the algebraic subgroups of $\Gm$ are given by equations $x^l=1$. The corresponding ideal of $\kkk\left[x,x^{-1}\right]$ is $(x^l-1)$.
    In the case of $C$-$\Sigma_1$-groups, where $$\Sigma_1 = \Z/t_1\Z\oplus\ldots\oplus\Z/t_p\Z =: \{\id=\alpha_1,\ldots,\alpha_t\},\quad t := t_1\cdot\ldots\cdot t_p,$$   there are more $\Sigma_1$-algebraic subgroups of ${\Gm}_{,\Sigma_1}$.
Let $C$ be an arbitrary Noetherian $\Sigma_1$-pseudofield. Let  also $\{e_0,\ldots,e_{s-1}\}$ be all
    indecomposable idempotents of $C$ with $\alpha_i(e_0) = e_{i-1}$, $1\Le i\Le s.$
    Then the $\Sigma_1$-Hopf algebra of ${\Gm}_{,\Sigma_1}$ is
    \begin{align*}
    C\{x,1/x\}_{\Sigma_1}=(K\times\ldots\times
    K)[x_\alpha,1/x_\alpha\:|\:\alpha\in\Sigma_1],
 \end{align*}
    where $K=C/\mathfrak{m}$ for a maximal ideal $\mathfrak{m}$ of $C$.
    We have 
    \begin{align*}
    C\{x,1/x\}_{\Sigma_1}&=e_0C\{x,1/x\}_{\Sigma_1}\times\ldots\times
    e_{s-1}C\{x,1/x\}_{\Sigma_1},\\
    R_i&=e_iC\{x,1/x\}_{\Sigma_1}=K[x_\alpha,1/x_\alpha\:|\:\alpha\in\Sigma_1].
    \end{align*}
    As we can see, each $R_i$ is a Hopf algebra. Let $I$ be the
    $\Sigma_1$-ideal defining our $\Sigma_1$-closed subgroup of
    ${\Gm}_{,\Sigma_1}$. Then, $$I=e_0I\times\ldots\times e_{s-1}I.$$ For each $i$, $0\Le i\Le s-1$, the ideal
    $e_iI\subset R_i$ is defined by equations
    \begin{align*}
        x_{\alpha_1}^{k_{i,1,\alpha_1}}\cdot\ldots\cdot x_{\alpha_t}^{k_{i,1,\alpha_t}}&=1,\\
        \vdots &\\
         x_{\alpha_1}^{k_{i,m,\alpha_1}}\cdot\ldots\cdot x_{\alpha_t}^{k_{i,m,\alpha_t}}&=1.
    \end{align*}
    So, if we collect all equations of all ideals $e_i I$, $0\Le i\Le s-1$,
    we obtain the equations
$$
    \begin{array}{rcl}
    e_0x^{k_{0,1,1}}\alpha_2\left(x^{k_{0,1,2}}\right)\cdot\ldots\cdot\alpha_t\left(x^{k_{0,1,
    t}}\right)&=&e_0,\\
    &\vdots&\\
    e_{s-1}x^{k_{s-1,m,1}}\alpha_2\left(x^{k_{s-1,m,2}}\right)\cdot\ldots\cdot\alpha_t\left(x^{k_{s-1,m,
    t}}\right)&=&e_{s-1}.
    \end{array}
$$
    Applying $\alpha_i^{-1}$ to the equations with $e_i$, $0\Le i \Le s$, we can rewrite the above system in the form
    \begin{equation}\label{eq:systemforGm}
    \begin{array}{rcl}
    e_0x^{k_{1,1}}\alpha_2\left(x^{k_{1,2}}\right)\cdot\ldots\cdot\alpha_t\left(x^{k_{1,
    t}}\right)&=&e_0,\\
    &\vdots&\\
    e_{0}x^{k_{m,1}}\alpha_2\left(x^{k_{m,2}}\right)\cdot\ldots\cdot\alpha_t\left(x^{k_{m,
    t}}\right)&=&e_{0},
    \end{array}
    \end{equation}
    which generate $I$ as a $\Sigma_1$-ideal.
    The latter equations also give generators of the ideal $e_0 I$. So, by   \cite[Section~2.2]{Waterhouse} we
    must have $m\Le t$.

    Now we claim that there is an equation in
    $I$ of the form $\varphi(x)-1=0$, where
    $\varphi(xy)=\varphi(x)\varphi(y)$. Indeed,  for this, denote the first equation in~\eqref{eq:systemforGm} by $\psi(x)-e_0$. Then, the
    equation
    $$
    \sum_{1\Le k\Le s}\alpha_k(\psi(x)-e_0)=\sum_{1\Le k\Le s}\alpha_k(\psi(x)) - 1.
    $$
    is of the desired form, where the sum $\sum_{1\Le k\Le s}\alpha_k(\psi(x)) $ is multiplicative
    because the $e_i$'s are orthogonal.

    Now suppose that $s=t$ (this is the case, for example, when $C$ is
    $\Sigma_1$-closed). In this case, we know that the number $m$ of equations does not exceed the number $s$ of our idempotents. Then the following system %of equations
    defines the ideal $I$.
   \begin{align}
    e_0x^{k_{1,1}}\alpha_2(x^{k_{1,2}})\cdot\ldots\cdot\alpha_t(x^{k_{1,
    t}})&=e_0,\tag{1}\\
    \vdots&&\notag\\
   e_0x^{k_{m,1}}\alpha_2(x^{k_{m,2}})\cdot\ldots\cdot\alpha_t(x^{k_{m,
    t}})&=e_0,\tag{m}\\
    e_0&=e_0,\tag{m+1}\\
    \vdots&&\notag\\
    e_0&=e_0.\tag{t}
    \end{align}
    Applying $\alpha_i$ to the $i$th equation, $1\Le i\Le t$, we obtain
    \begin{align}
    e_0x^{k_{1,1}}\alpha_2(x^{k_{1,2}})\cdot\ldots\cdot\alpha_t(x^{k_{1,
    t}})&=e_0,\tag{1}\\
    &\vdots\notag\\
    e_{m-1}\alpha_m(x^{k_{m,1}})(\alpha_m\alpha_2)(x^{k_{m,2}})\cdot\ldots\cdot(\alpha_m\alpha_t)(x^{k_{m,
    t-1}})&=e_{m-1}\tag{m},\\
   e_{m}&=e_{m}\tag{m+1},\\
    &\vdots\notag\\
    e_{t-1}&=e_{t-1}\tag{t}.
    \end{align}
   By taking the sum of the above equations, we arrive at an equation of
    the form
    \begin{equation}\label{eq:phi1}
    \varphi(x)=1.
    \end{equation}
     Since the $e_i$'s are orthogonal, the left-hand
    side is multiplicative. Moreover, this equation defines
    the same subgroup. Vice versa, every multiplicative $\varphi(x) \in C\{x,1/x\}_{\Sigma_1}$ defines a $\Sigma_1$-subgroup of ${\Gm}_{,\Sigma_1}$ via~\eqref{eq:phi1}. Note  that it might happen that  the set of solutions is empty. For example, this is the case for $\varphi=e$, where $e$ is idempotent and not equal to $1$.
\end{example}

\begin{example}Let $C = \C\times \C\times \C$ with
$
\rho(a_0,a_1,a_2) = (a_2,a_0,a_1),$ $a_i \in \C$.
By \cite[Proposition 15]{Dima}, $(C,\rho)$ is a $\Sigma_1$-closed pseudofield. Let
\begin{equation}\label{eq:defG}
G = \{a\in C\:|\: a\cdot\rho(a) = 1\},
\end{equation}
a $\Sigma_1$-subgroup of $\Gm$ considered in Example~\ref{ex:subgroupsofGm}.
A calculation shows that
$
G = \{(1,1,1),\:(-1,-1,-1)\}$.
This demonstrates  a major difference between $\Sigma_1$-subgroups and differential
algebraic subgroups (see \cite[Chapter~IV]{Cassidy}) of $\Gm$. More precisely,
in the differential case the order of the defining equation coincides with the algebraic dimension of the subgroup.

 In our case, the order of $\rho$ in~\eqref{eq:defG} is equal to $1$, however, the group is finite. Therefore, in order to compute the algebraic dimension of a $\Sigma_1$-group one needs to do more calculation than just to look at the $\rho$-order of the equation.
\end{example}
\subsection{Galois group}
As before, let $K$ be a Noetherian $\Sigma$-pseudofield with $C := K^\sigma$ being $\Sigma_1$-closed.
\begin{definition} Let $L$ be a Picard--Vessiot pseudofield of equation~\eqref{eq:diffeq}. Then the group of $\Sigma$-automorphisms of $L$ over $K$ is called the difference Galois group of~\eqref{eq:diffeq} and denoted by $\Aut_{\Sigma}(L/K)$.
\end{definition}
Let $L$ be a Picard--Vessiot pseudofield of equation~\eqref{eq:diffeq} and $F \in \GL_n(L)$ be a fundamental matrix. Then for any $\gamma \in \Aut_\Sigma(L/K)$ we have
\begin{equation}\label{eq:gamma}
\gamma(F) = FM_\gamma,
\end{equation} where $M_\gamma \in \GL_n(C)$, which, as usual, defines an injective group homomorphism from $\Aut_\Sigma(L/K)$ into $\GL_n(C)$. Since $L$ is generated by the entries of $F$, the action of $\gamma$ on $L$ is determined by its action on $F$. This induces an identification of $\Aut_\Sigma(L/K)$ with $\Aut_\Sigma(R/K)$,
where $R$ is the Picard--Vessiot ring corresponding to $F$.

We will now construct a map
$
\Aut_\Sigma(R/K) \to \Max_\Sigma(R\otimes_KR)$, the maximal $\Sigma$-ideals of $R\otimes_KR$.
For this, let $F$ be a fundamental matrix of equation~\eqref{eq:diffeq} with entries in $R$ and $\gamma \in \Aut_\Sigma(R/K)$. As above, $\gamma F = FM_\gamma$, where $M_\gamma \in \GL_n(C)$. We will then map
$$
\gamma \mapsto {[F\otimes 1-1\otimes FM_\gamma]}_\Sigma,$$
the smallest $\Sigma$-ideal containing $F\otimes 1-1\otimes FM_\gamma$.
Since $R$ is $\Sigma$-simple, the kernel of the surjective  $\Sigma$-homomorphism
$$(\gamma,\Id)\colon R\otimes_K R\to R,$$ which is $[F\otimes 1-1\otimes FM_\gamma]_\Sigma$, is a maximal $\Sigma$-ideal in $R\otimes_K R$.

To construct a map in the reverse direction, let $$
\phi_1,\phi_2 : R \to R\otimes_KR,$$ with $r\mapsto r\otimes 1$ and $r\mapsto 1\otimes r$, respectively.
Let $\m$ be a maximal $\Sigma$-ideal of $R\otimes_KR$. Then, $(R\otimes_KR)/\m$ is a Picard--Vessiot ring of equation~\eqref{eq:diffeq}. As in Proposition~\ref{prop:PVringunique}, the composition homomorphisms
$$
\overline{\phi}_i : R \to R\otimes_KR \to (R\otimes_KR)/\m
$$
are isomorphisms. This induces an automorphism of the ring $R$ defined by
$$
\phi_{\m} := \overline{\phi}_2^{-1}\circ\overline{\phi}_1.$$

\begin{proposition}\label{prop:MaxAut} The correspondence $\Aut_\Sigma(R/K) \to \Max_\Sigma(R\otimes_KR)$ constructed above is bijective. Moreover, these bijections are inverses of each other.
\end{proposition}
\begin{proof} Let $\gamma\in \Aut_\Sigma(R/K)$ and $M \in \GL_n(C)$ be such that $\gamma(F) = FM$. Set $\m = [F\otimes 1 - 1\otimes FM]_\Sigma$.
Since $$\overline{\phi}_1(F) = F\otimes 1,\quad F\otimes 1  = 1\otimes FM\ \text{in}\ (R\otimes _KR)/\m,\ \text{and}\ \ \overline{\phi}_2(FM) = 1\otimes FM,$$ we have $\phi_{\m}(F)=FM$.
%That is,
%$$
%\xymatrix@R=8pt@C=12pt{
 %   &(R\mathop{\otimes}_K R)/\frak m&\\
 %   R\ar[ur]^-{\bar\phi_1}\ar[rr]_{\phi_{\frak m}}&&R\ar[ul]_-{\bar\phi_2}\\
%}
%$$

Conversely, let $\m \in \Max_\Sigma(R\otimes_KR)$. Then $\phi_{\m}(F) = FM$ for some $M \in\GL_n(C)$. Hence, $$\overline{\phi}_1(F) = \overline{\phi}_2(FM).$$ Thus, $$
[F\otimes1-1\otimes FM]_{\Sigma} \subset \m.$$
Since, as above, the former ideal is $\Sigma$-maximal, it coincides with $\m$.
 \end{proof}
\begin{proposition}\label{prop:isomCG} The Galois group $G$ of equation~\eqref{eq:diffeq} is a closed subgroup of $\GL_n(C)$. Moreover, if the ring $R\otimes_KR$ is reduced, then
$$
R\otimes_KR \cong R\otimes_CC\{G\},$$
where $C\{G\}$ is the ring of regular functions on $G$ and $R$ is a Picard--Vessiot ring of~\eqref{eq:diffeq}.
\end{proposition}
\begin{proof}
As before, define $\sigma$ on the $\Sigma_1$-ring
$
R\{X_{ij},1/\det X\}_{\Sigma_1}$
by
$\sigma X = AX$. Let $F$ be a fundamental matrix of~\eqref{eq:diffeq} with coefficients in $R$ and let, as above, $Y = F^{-1}X$,
which implies that $\sigma Y = Y$. We have a $\Sigma$-isomorphism
$$
R\otimes_KK\{X_{ij},1/\det X\}_{\Sigma_1} \cong R\otimes_CC\{Y_{ij},1/\det Y\}_{\Sigma_1}.
$$
As in the proof of Proposition~\ref{prop:LQtR}, this induces a $\Sigma$-isomorphism
\begin{equation}\label{eq:RBisomorphism}
R\otimes_KR \cong R\otimes_C B,
\end{equation}
where $B = C\{Y_{ij},1/\det Y\}_{\Sigma_1}/J$ and $J$ is a $\Sigma_1$-ideal.

By Proposition~\ref{prop:MaxAut}, $\Aut_\Sigma(R/K)$ as a set can be identified with $\Max_\Sigma(R\otimes_KR)$. The latter set, by Proposition~\ref{prop:idealcorrespondance} and isomorphism~\eqref{eq:RBisomorphism}, can be identified with $\Max_{\Sigma_1}B$. Since $C$ is $\Sigma_1$-closed,  by \cite[Proposition~14]{Dima}, the set $\Max_{\Sigma_1} B$ can be identified with
a closed subset of $\GL_n(C)$.  The group structure of $G$ is preserved under this identification due to~\eqref{eq:gamma}. If the ring $R\otimes_KR$ is reduced, then the ideal $J$ is radical and, therefore, $B$ is the coordinate ring of $G$.
 \end{proof}

\subsection{Galois correspondence}
\begin{proposition}\label{prop:LGK} Let $L$ be a Picard--Vessiot pseudofield of equation~\eqref{eq:diffeq}, $R$ be its Picard--Vessiot ring, and $G$ be its Galois group. If the ring $R\otimes_K R$ is reduced, then
$
L^G = K$.
\end{proposition}
\begin{proof}
Let
\begin{equation}\label{eq:abnotinK}
a/b \in L\setminus K,
\end{equation} where $a$, $b \in R$ and $b$ is not a zero divisor. Set
$$
d = a\otimes b - b\otimes a \in R\otimes_KR.$$
We will show that $d \ne 0$. For this, let $\{e_1,\ldots,e_n\}$ be all indecomposable idempotents of the Noetherian $\Sigma$-pseudofield $K$. Since $b$ is not a zero divisor,
\begin{equation}\label{eq:eibne0}
e_i\cdot b \ne 0,\quad 1\Le i\Le n.
\end{equation}
Suppose that for each $i$, $1\Le i\Le n$, $e_i\cdot a$ and $e_i\cdot b$ are linearly dependent over $e_iK$, that is,
$
\lambda_i\cdot e_i\cdot a=\mu_i\cdot e_i\cdot b$
for all $i$. Then~\eqref{eq:eibne0} implies that $\lambda_i \ne 0$, $1\Le i \Le n$. Since $e_iK$ is a field, we have
$
e_i\cdot a = e_i\cdot b\cdot\mu_i/\lambda_i$.
Hence,
$$
a =\sum_{i=1}^ne_ia=\left(\sum_{i=1}^n\frac{\mu_i}{\lambda_i}e_i\right)b,\quad \text{that is},\quad
\frac{a}{b} = \sum_{i=1}^n\frac{\mu_i}{\lambda_i}e_i \in K,
$$
which is a contradiction to~\eqref{eq:abnotinK}. Therefore, there exists $i$, $1\Le i\Le n$, such that
$e_i\cdot a$ and $e_i\cdot b$ are linearly independent over $e_iK$.
% If $d =0$ in $R\otimes_KR$ then
Then,
%\begin{equation}\label{eq:eiab}
$$e_i\cdot a\otimes e_i\cdot b-e_i\cdot b\otimes e_i\cdot a \ne 0$$ in $e_iR\otimes_{e_iK}e_iR$.
%\end{equation}
%Indeed, in general,  if $A$ is a ring and $B$, $C$, and
%$D$ are $A$-algebras then
%$$
%(B\otimes_A C)\otimes_A D\iso B\otimes_A D
%\otimes_A  C\iso B\otimes_A D\otimes_D D
%\otimes_A C\iso (B\otimes_A D)\otimes_D (C
%\otimes_A D).
%$$
%Moreover, we have the following commutative diagram
%$$
%\xymatrix@R=4pt@C=6pt{
 %   &B\mathop{\otimes}_AC&\\
%    B\ar[ru]&&C\ar[lu]\\
 %   &A\ar[lu]\ar[ru]&\\
%}
%$$
%Hence, for any element in $B$ (or $C$) one can take its image in
%$B\otimes_A C$. So, we choose $A = K$, $B=C = R$, and $D = e_iK$
%to obtain~\eqref{eq:eiab}
%contradicting that $1\otimes e_i\cdot a$ and $1\otimes
%e_i\cdot b$ are linearly independent over $e_iR$.
Hence,
\begin{equation*}\label{eq:abba}
 a\otimes b - b\otimes a \ne 0\quad\text{in}\quad R\otimes_KR.
\end{equation*}
We will now show that there is a maximal $\Sigma$-ideal in $R\otimes_K R$ that does not contain $d$. Since $R\otimes_KR$ is reduced, then by Proposition~\ref{prop:isomCG} we have
$
R\otimes_KR \cong R\otimes_CC\{G\}$.
Let $\{l_i\}_{i\in\mathcal{I}}$ be a basis of $R$ over $K$. Then there exist $r_1,\ldots,r_m \in C\{G\}$ such that
$$
d = l_1\otimes r_1+\ldots l_m\otimes r_m.$$
Since $r_1$ is not nilpotent, there exists a maximal $\Sigma_1$-ideal $\m \subset C\{G\}$ such that $\overline{r}_1\ne 0$ in $C\{G\}/\m$. Then image of $d$ in $R\otimes_C C\{G\}/\m \cong R$ is
$$
\overline{d} = l_1\overline{r}_1+\ldots +l_m\overline{r}_m.$$
Since $\overline{r}_1\ne0$, we have $\overline{d} \ne 0$. Thus, $d \notin R\otimes_C\m$.
Using the correspondence between maximal $\Sigma$-ideals in $R\otimes_KR$
and $\Sigma$-automorphisms of $R$ over $K$, let
$$
\phi_{\m} = \overline{\phi}_2^{-1}\circ\overline{\phi}_1$$
correspond to $\m$ as in the proof of Proposition~\ref{prop:MaxAut}. Then our
choice of $\m$ implies that
\begin{equation}\label{eq:ne0}
 (R\otimes_K R)/\m \ni
\overline{\phi}_1(a)\overline{\phi}_2(b)-\overline{\phi}_1(b)\overline{\phi}_2(a)
\ne 0.
\end{equation}
Applying $\overline{\phi}_2^{-1}$ to both sides of~\eqref{eq:ne0},
we obtain that
$
\phi_{\m}(a)b-\phi_{\m}(b)a \ne 0$.
Therefore, $\phi_{\m}\left(\frac{a}{b}\right) \ne \frac{a}{b}$.
 \end{proof}

\begin{lemma}\label{lem:transitive} Let $K\subset L$ be Noetherian $\Sigma$-pseudofields. Let $H \subset \Aut_\Sigma(L)$ such that $L^H=K$. Suppose that
$
K \cong \prod_{i=1}^n F_{\Sigma_1}(F)
$
as $\Sigma_1$-rings, where $F$ is a field. Let $\{e_i\}$ be the corresponding idempotents. Then for each $i$ the abstract group generated by $\Sigma_1$ and $H$ acts transitively on the set of indecomposable idempotents of the ring $e_iL$.
\end{lemma}
\begin{proof} Let $e \in e_iL$ be an idempotent and $S$ be its $\Sigma_1 * H$-orbit. The set $S$ coincides with the set of indecomposable idempotents if and only if
$
\sum_{f \in S} f = 1$.
This sum is $H$-invariant and, therefore, it belongs to $F_{\Sigma_1}(F)$. Since it is $\Sigma_1$-invariant as well, it is equal to $1$, because a $\Sigma_1$-invariant
idempotent of $F_{\Sigma_1}(F)$ generates a $\Sigma_1$-ideal.
 \end{proof}

\begin{proposition}\label{prop:HG} Let $L$ be a Picard--Vessiot pseudofield for equation~\eqref{eq:diffeq} and $H$ be a closed subgroup of the Galois group $G$. Then $L^H = K$ implies $H = G$.
\end{proposition}
\begin{proof}
As before, let $F$ be a fundamental matrix with entries in $L$ and $\sigma X = AX$ define the action of $\Sigma$ on the $\Sigma_1$-ring $D := L\{X_{ij},1/\det X\}_{\Sigma_1}$. Let also $Y = F^{-1}X$. Again, as before,
$$
L\otimes_KK\{X_{ij},1/\det X\}_{\Sigma_1} \cong L\otimes_CC\{Y_{ij},1/\det Y\}_{\Sigma_1}.
$$
Suppose that $H \subsetneq G$ and let $I\subsetneq J$ be the defining ideals of $G$ and $H$, respectively. Denote their extensions to $L\{X_{ij}, 1/\det X\}$ by $(I)$ and $(J)$, respectively. By Proposition~\ref{prop:idealcorrespondance}, we have
$
(I) \subsetneq (J)$.
Explicitly, we have
$$
(I) = \left\{f(X) \in L\{X_{ij},1/\det X\}_{\Sigma_1}\:|\: f(FM)=0\ \text{for all}\ M\in G\right\}
$$
and
\begin{equation}\label{eq:idealJ}
(J) = \left\{f(X) \in L\{X_{ij},1/\det X\}_{\Sigma_1}\:|\: f(FM)=0\ \text{for all}\ M\in H\right\}.
\end{equation}
Let $T = (J)\setminus (I) \ne \varnothing$.
Define the action of $H$ on $L\otimes_KK\{X_{ij},1/\det X\}_{\Sigma_1}$ by
$$
h(a\otimes b) = h(a)\otimes b,\quad h\in H.$$
Then, equality~\eqref{eq:idealJ} implies that $(J)$ is stable under this action of $H$. By Proposition~\ref{prop:product},
$$
K \cong F_{\Sigma_1}(F)\times\ldots\times F_{\Sigma_1}(F)
$$
as $\Sigma_1$-rings, where $F$ is a field. Let $e_1,\ldots, e_n$ be the idempotents corresponding to the components $F_{\Sigma_1}(F)$ in the above product. By Proposition~\ref{prop:26},  the ring $K\{X_{ij},1/\det X\}_{\Sigma_1}$ has a $\Sigma_1$-invariant basis $\{Q_\alpha\}$. Then every element of the ring $D$ is of the form
\begin{equation}\label{eq:Qsum}
Q = q_1Q_{\alpha_1}+\ldots+q_nQ_{\alpha_n},
\end{equation}
where $q_i \in L$, $1\Le i\Le n$. Let $Q$ be an element in $T$ with the shortest presentation of the form~\eqref{eq:Qsum}. Since
$
Q =\sum_i e_iQ$,
there exists $i$ such that $e_iQ \in T$. Denote the latter polynomial by $Q$ as well. Now, we have $Q \in e_iD$.
Let $\{f_1,\ldots, f_m\}$ be all indecomposable idempotents of the Noetherian ring $e_iL$. Then,
$$
Q = \sum\nolimits_{j=1}^mf_jQ.$$
Hence, there exists $j$ such that $f_jQ \in T$. By Lemma~\ref{lem:transitive}, there exist $h_t \in \Sigma_1 * H$ such that the coefficients of
$$
Q' := \sum_th_t(Q)$$
are invertible in $e_iL$. Therefore,
$$
Q' = e_iQ_1+g_2Q_2+\ldots+g_mQ_m.$$
Since the ideal $(J)$ is stable under the action of
 $\Sigma_1*H$, we have $Q'
\in T$. Since $e_i \in K$, for every $h \in H$ the polynomial
$
Q''_h := Q'-h(Q')$
has a shorter presentation than $Q$ and, therefore, $Q''_h \notin T$. That is,
\begin{equation}\label{eq:QinI}
Q''_h \in (I)\quad\text{for all}\quad h \in H.
\end{equation}
We will show now that $Q''_h = 0$ for all $h \in H$. Suppose that
$Q''_h \ne 0$ for some $h \in H$. Then~\eqref{eq:QinI} implies that
there exists $j$ such that $0 \ne f_jQ''_h \in (I)$. Since
$\Sigma_1*H$ acts
transitively on the indecomposable idempotents of $e_iL$, there
exist $\phi_t \in
\Sigma_1*H$  such that
$$
\overline{Q}_h := \sum_t \phi_t\left(Q''_h\right) = r_2Q_2+\ldots+r_mQ_m \in (I),$$
where $r_2$ is invertible in $e_iL$. Therefore, there exists $r \in e_iL$ such
that $g_2 = rr_2$. Then, the polynomial $
Q'-r\overline{Q} \in T$
has a shorter presentation than $Q'$, which is a contradiction.
We have shown that $h(Q') = Q'$
for all $h \in H$. Hence, all coefficients of $Q'$ are in $K$ and, therefore,
are invariant under the action of $G$ as well. Since $0 = Q'(F\cdot \id) = Q'(F)$, we have
$$
0 = g(Q'(F)) = g(Q')(FM_g)=Q'(FM_g)
$$
for all $g \in G$. Thus, $Q' \in (I)$, which contradicts to $Q' \in T$.
 \end{proof}

\begin{lemma}\label{lem:DFH}
Let $M$ be a field,
\begin{equation}\label{eq:DLL}
D := M\times\ldots\times M,
\end{equation} $F \subset D$ be a subfield and $H \subset \Aut(D)$ with $D^H = F$. Let $f := (1,0,\ldots,0) \in D$ and $H_1 \subset H$ be the stabilizer of $f$. Then,
$
fF = M^{H_1}$,
where $M$ is from the first component in~\eqref{eq:DLL}.
\end{lemma}
\begin{proof}
Since $fF$ is $H_1$-invariant, we have
$fF\subseteq M^{H_1}$.
We will show  the reverse inclusion. Let $l\in (fD)^{H_1} = M^{H_1}$. We need to show that there is an element $a\in F$ such that
$l=fa$.
Let the $H$-orbit  of $l$ be $\{l_1,\ldots,l_k\}$, where $l=l_1$. For each $i$, $1\Le i \Le k$, there exists $a_i \in D$ such that
$l_i = a_if_i$, where $f_i$ is the
idempotent corresponding to the $i$th factor in $D$ (so we have $f=f_1$), since if $l\ne 0$, then $H_1$ is
the stabilizer of $l$. Hence, for $d=\sum_{i=1}^k l_i$ we have $$fd=\sum_{i=1}^k f_1l_i=l_1=l$$
and $H$ permutes the $l_i$'s. Thus, $d\in D^H = F$ as desired.
 \end{proof}

\begin{proposition}\label{prop:equivconditions} Let $K$ be a Noetherian $\Sigma$-pseudofield, $R$ be a $\Sigma$-simple Noetherian algebra over $K$, and $L = \Qt(R)$. Then for the statements
\begin{enumerate}
\item\label{i1} the ring $R\otimes_KR$ is reduced,
\item\label{i2} the ring $L\otimes_KL$ is reduced,
\item\label{i3} there exists a subgroup $H \subset \Aut_\Sigma(L/K)$ such that $L^H=K$,
\end{enumerate}
we have: \eqref{i1} is equivalent to~\eqref{i2} and~\eqref{i3} implies~\eqref{i2}. Moreover, if $R$ is a Picard--Vessiot ring over $K$, then the above statements are equivalent.
\end{proposition}
\begin{proof} The equivalence of~\eqref{i1} and~\eqref{i2} follows from the fact that $R\otimes_KR \subset L\otimes_KL$ and that the latter ring is a localization of the former one.
We will show that~\eqref{i3} implies~\eqref{i2}. Let $\{e_1,\ldots,e_n\}$ be the indecomposable idempotents of $K$. Then,
$$
L\otimes_KL = \prod_{i=1}^ne_iL\otimes_{e_iK}e_iL.$$
%Indeed, $A=A_1\times
%A_2$ be a ring and $B$ and $C$ be $A$ algebras. Denote by $e$ and
%$f$ the idempotents $(1,0)$ and $(0,1)$ of $A$, respectively. Then, we
%have decompositions $B=eB\times fB$ and $C=eC\times fC$. We will show now that
%$$B\otimes_A C = eB\otimes_{eA}eC\times fB\otimes_{fA}fC.$$
%For this, first note that $eB\otimes_A fC=0$.
%Indeed, $eb\otimes fc=e(eb)\otimes fc=eb\otimes e(fc)=0$. Hence,
%$$B\otimes_A C = (eB\oplus fB)\otimes_A (eC\oplus fC)=eB\otimes_A
%eC\oplus fB\otimes_A fC,$$ Since the
%homomorphism $A\to eB\otimes_{A}eC$ factors through $eA$, we have
%$eB\otimes_AeC=eB\otimes_{eA}eC$.
It is enough to show that the ring $e_iL\otimes_{e_iK}e_iL$ is reduced. Note that $e_iK$ is a field. Since $e_i \in K$, they are all invariant under $H$ and, moreover, $
(e_iL)^H=e_iK$.
Let now $\{f_1,\ldots, f_m\}$ be the indecomposable idempotents of the ring $e_iL$ and let $H_1$ be the stabilizer of $f_1$. Lemma~\ref{lem:DFH} with $D=e_iL$ and  $F=e_iK$ implies that
$$
(e_if_1L)^{H_1}=f_1e_iK.
$$
Since
$$
e_iL\otimes_{e_iK}e_iL = \prod_{s,t}e_if_sL\otimes_{e_iK}e_if_tL,
$$
it remains to show that the ring
$$D := e_if_sL\otimes_{e_iK}e_if_tL$$ is reduced. By
\cite[Corollary~1, \S7, no.~2]{Bourbaki}, with $A = e_if_sL$,
$B = e_if_tL$, $N = B$, and $K=e_iK$, the Jacobson radical of the
ring $D$ is zero. In particular, the ring $D$ is reduced.

The last statement follows from Proposition~\ref{prop:LGK}.
 \end{proof}
\begin{definition} A Picard--Vessiot extension $L/K$ is called separable if one of the three equivalent conditions in Proposition~\ref{prop:equivconditions} is satisfied.
\end{definition}
\begin{theorem}\label{thm:324} Let $R$ be a Picard--Vessiot ring of equation~\eqref{eq:diffeq} and $L=\Qt(R)$ be separable over $K$. Let $\F$ denote the set of all
intermediate $\Sigma$-pseudofields $F$ such that $L$ is separable over $K$
and $\Grp$ denote the set of all $\Sigma_1$-closed subgroups $H$ in the Galois group $G$ of $L$ over $K$. Then the correspondence
$$
\F \longleftrightarrow \Grp,\quad F\mapsto \Aut_\Sigma(L/F),\quad H\mapsto L^H
$$
is bijective and the above maps are inverses of each other. Moreover, $H$ is normal in $G$ if and only if the $\Sigma$-pseudofield $F := L^H$ is $G$-invariant.
 \end{theorem}
\begin{proof}
The map $\F \longrightarrow \Grp$ is well-defined by Proposition~\ref{prop:isomCG}. Propositions~\ref{prop:1} and~\ref{prop:2} imply that $L^H\subset L$ is a $\Sigma$-pseudofield. By Proposition~\ref{prop:Noetherian}, it is Noetherian and, by Proposition~\ref{prop:equivconditions}, it is separable.

Let $F \in \F$. Then the extension $L$ over $F$ is separable and is a Picard--Vessiot pseudofield for equation~\eqref{eq:diffeq} considered over $F$. Moreover,
$
F = F^{\Aut_\Sigma(L/F)}
$
by Proposition~\ref{prop:LGK}.
Conversely, let $H$ be a $\Sigma_1$-closed subgroup of $G$. Set $F = L^H$. Then $L$ is a Picard--Vessiot pseudofield for equation~\eqref{eq:diffeq} over
$F$. By Proposition~\ref{prop:HG}, we have $H = \Aut_\Sigma(L/F)$.
The equality
$$
g(F) = \left\{r \in L\:|\: ghg^{-1}r=r\ \text{for all}\ h\in H\right\}
$$
implies the statement about  normality.
 \end{proof}

\begin{remark}
The base pseudofield $K$ is a product of the fields, say
$L\times\ldots\times L$. If the field $L$ is perfect, then for every
pseudofields $F$ and $E$ containing $K$ the ring $F\otimes_K E$ is
reduced.
Indeed, let
$e_0,\ldots,e_{t-1}$ be all indecomposable idempotents of $K$, then
$$
F\otimes_K E=\prod_{i=0}^{t-1} e_iF\otimes_{L} e_i E.
$$
Since $L$ is perfect and $L$-algebras $e_iF$ and $e_i E$ are
reduced, then $e_iF\otimes_{L} e_i E$ is reduced as well
(see~\cite[A.V.~119, No.~5, Th\'eor\`em~3(d)]{Bourbaki2}). Therefore,
if $L$ is perfect, then any Picard--Vessiot extension is separable.
If the field $L$ is finite, algebraically closed or of
characteristic zero, then $L$ is perfect. In this case, the set $\F$
contains all intermediate $\Sigma$-pseudofields.
\end{remark}

\subsection{Extension to non-faithful action}\label{sec:nonfaithful}
In the introduction, we alluded to the fact that some of our results could instead be
obtained via faithfully flat descent  \cite{MilneEtale}. However, this requires the additional assumption\footnote{We also assume this in Section~\ref{sec:PVTheory},
but it is only needed to guarantee the uniqueness of a parameterized PV-extension, which we do not use in the applications.
It also implies the existence, but is only a sufficient condition, and there are situations when one does not have to make this assumption.
The Galois correspondence in its Hopf-algebraic version does not need this as we show below.} 
that $\Sigma_1$ acts faithfully in the setup for faithfully flat descent. The theory we have
developed in this paper works more generally, as we now illustrate.
Let $L$ be a
field, $\Sigma_1=\mathbb Z/4\mathbb Z$, and $\rho$ be a generator.
Suppose that $\Sigma_1$ acts on $L$ faithfully and $K=L^{\rho^2}$. Then
$$
\mathrm{Aut}^\rho(L/K)\cong\mathbb Z/2\mathbb Z\quad \text{but}\quad
\mathrm{Aut}{\left(L^{\Sigma_1}/K^{\Sigma_1}\right)}=\{1\},
$$
where $\mathrm{Aut}^\rho$ is the set of $\rho$-automorphisms (so, we also have
to store the order of $\rho$ -- not only the field of invariants).
Roughly speaking, replacing a difference object by a non-difference
one, we cannot recover the group of difference automorphisms. This
example (where $L$ is a Picard--Vessiot extension and $K$ is a base
field) appears naturally in the Picard--Vessiot theory if, for instance, the initial
field contains only $\rho$-constants. Also note that if we replace
$L$ by $L^{\Sigma_1}$, then the $\sigma$-constants of $L^{\Sigma_1}$ are not
necessarily algebraically closed. Therefore, using descent, extra preparatory steps
are required  before we are able to apply the standard
non-parametric difference Galois theory.

Here is another example in which it is preferable to consider not necessarily faithful actions of $\Sigma_1$. Let 
$$L = \mathbb{C}(x),\quad  \sigma(x) = 2\cdot x,\ \rho(x) = -x,\quad K = \mathbb{C}{\left(x^2\right)},\quad \Sigma_1 = \{\mathrm{id},\rho\}$$ and the difference equation
be 
\begin{equation}\label{eq:1}
\sigma(y) = 2\cdot y.
\end{equation} There are no solutions of~\eqref{eq:1} in $K$
and $L$ is a Picard--Vessiot extension of the equation, on which $\Sigma_1$ acts faithfully, while the action of $\Sigma_1$ on $K$ is trivial. Of course, the field $L$ considered with the trivial action of $\rho$ is also a Picard--Vessiot extension, but restricting to this would not allow us to consider more interesting and useful cases outlined in this example.  We will show how one can generalize our results to include non-faithful actions of $\Sigma_1$.
\begin{theorem}[(Instead of Theorem~\ref{thm:324})]\label{thm:1new}
Let $R$ be a Picard--Vessiot ring over a pseudofield $K$ with the
corresponding $\Sigma_1$-Hopf-algebra $H$ (which replaces the Galois group -- see \cite[Section~2]{HopfGalois}) and
$L=\Qt(R)$ be the pseudofield of fractions. Let
$\mathcal F$ denote the set of all intermediate
$\Sigma$-pseudofields and $\mathcal G$ denote the set of all
$\Sigma_1$-Hopf-ideals in $H$. Then the correspondence
$$
\mathcal G \rightarrow \mathcal F,\:\: I\mapsto L^I := \{x \in L\:|\: 1\otimes x-x\otimes 1 \in I\cdot (L\otimes_K L)\}
$$
is bijective. Moreover, $I$ is normal in $H$ if and only if the
$\Sigma$-pseudofield $L^I$ is $H$-invariant. A
$\Sigma_1$-Hopf-ideal $I$ is radical if and only if $L$ is separable
over $L^I$.
\end{theorem}

In order to prove this result, one extends the Hopf-algebraic approach
given in~\cite{Amano2}. The main technical results one uses are \cite[Proposition 3.10]{Amano2} and

\begin{proposition}[(Instead of Proposition~\ref{prop:idealcorrespondance})]
Let $R$ be a Picard-Ves\-siot ring over $K$ and $$R\otimes_K
R\cong R\otimes_C H$$ be the torsor isomorphism for $R$, where $H$ is a
$\Sigma_1$-Hopf-algebra over $\Sigma_1$-pseudofield $C=R^\sigma$.
Then this isomorphism induces a $1-1$ correspondence between
$\Sigma_1$-Hopf-ideals of $H$ and $\Sigma$-coideals of $R\otimes_K
R$. That is, if $I\subseteq R\otimes_K R$ is a $\Sigma$-coideal, then
$I\cap H$ is a $\Sigma_1$-Hopf-ideal; if $\frak a$ is a
$\Sigma_1$-Hopf-ideal, then $R\otimes_C \frak a$ is a
$\Sigma$-coideal.
\end{proposition}

From the Hopf-algebraic point of view, the Galois correspondence
can be derived from Theorem~\ref{thm:1new} above using the following result:

\begin{theorem}
Let $H$ be a reduced $\Sigma_1$-Hopf-algebra over a difference closed
pseudofield $K$. Then $H$ induces a functor $F$ from the category of
$\Sigma_1$-$K$-algebras to the category of groups by the rule
$$F(R)=\hom_{\Sigma_1-K}(H,R).$$ Then $H$ can be recovered from 
$F(K)$.
\end{theorem}

\subsection{Torsors}\label{sec:torsor}
Let $C$ be a $\Sigma_1$-closed pseudofield and $K\supset C$ be a Noetherian $\Sigma$-pseudofield. Let $G$ be a $\Sigma_1$-group over $C$ be $C\{G\}$ be its $\Sigma_1$-Hopf algebra with comultiplication $\Delta$, antipode $S$, and counit $\varepsilon$.
\begin{definition} A $\Sigma_1$-finitely generated $K$-algebra $R$ supplied with a $\Sigma$-$K$-algebra homomorphism $$\nu^*: R\to R\otimes_CC\{G\}$$
is called a $G$-torsor over $K$ if the following statements are true:
\begin{enumerate}
\item $R$ is a $C\{G\}$-comodule with respect to $\nu^*$,
\item\label{iso2} the vertical arrow in the following diagram is an isomorphism:
$$
\xymatrix@R=8pt@C=8pt{
    &R\mathop{\otimes}_K R\ar[dd]&\\
    R\ar[ur]^-{\id_R\otimes 1}\ar[dr]_-{\nu^*}&&R\ar[ul]_-{1\otimes \id_R}\ar[dl]^-{1\otimes \id_R}\\
    &C\{G\}\mathop{\otimes}_C R&\\
}
$$
%is an isomorphism.
\end{enumerate}
\end{definition}
In the above notation, the rings $R$ and $C\{G\}$ are finitely generated algebras over Artinian rings. Then the Krull dimension is defined for them, which we will denote by $\dim R$ and $\dim C\{G\}$, respectively. The isomorphism in~\eqref{iso2} implies that $
\dim R = \dim C\{G\}$.
Moreover, let $e$ be an indecomposable idempotent in $C$ and $F := eC$ be
the corresponding residue field. Then $F\otimes_CC\{G\}$ is a finitely generate $F$-algebra of dimension equal to $\dim C\{G\}$. Hence, for any minimal prime ideal $\p$ of the ring $F\otimes_CC\{G\}$,
$$
\trdeg_F \kkk(\p) = \dim C\{G\} = \dim R,
$$
where $\kkk(\p)$ is the residue field of $\p$.

\begin{proposition}\label{prop:Galoistorsor} Let $K$ be a Noetherian $\Sigma$-pseudofield with $K^\sigma$ being a $\Sigma_1$-closed pseudofield. Let $R$ be a Picard--Vessiot ring for equation~\eqref{eq:diffeq} with $L = \Qt(R)$. Let $G$ be the Galois group of $L$ over $K$. If $R$ is separable over $K$, then $R$ is a $G$-torsor over $K$.
\end{proposition}
\begin{proof} Follows from Proposition~\ref{prop:isomCG}.
 \end{proof}

\section{Applications}\label{sec:Applications}
We will start by giving a difference dependence statement in the spirit of \cite{CharlotteComp}, which we prove using our own methods. We then show how this is related to Jacobi's theta-function in Section~\ref{sec:Jacobi} and demonstrate our applications in Section~\ref{sec:generalqdifference}, in particular, in Theorem~\ref{thm:6}, which provides a very explicit difference dependence test.
\subsection{General approach}

For any nonzero complex number $a$ we define an automorphism  $\sigma_a :\C(z) \to \C(z)$ by  $$\sigma_a(f)(z)=f(az).$$ Let
$\Sigma_1\subseteq \mathbb C^*$ be a finite subgroup. Then
$\Sigma_1$ is a cyclic group generated by a root of unity
$\zeta$ of degree $t$.
Let $q\in\mathbb C$ be a complex number such that $|q|>1$. Now we
have an action of the group $\Sigma=\Z\oplus\Z/t\Z$ on $\C(z)$, where  the first summand is generated by
$\sigma_q$ and  the second one  is  generated by $\sigma_\zeta$.
Throughout this section the ring $\C(z)$ is supplied with
this structure of a $\Sigma$-ring.

\begin{theorem}\label{thm:app}
Let $R$ be a $\Sigma$-ring containing the field $\C(z)$ such that
$\kkk:=R^{\sigma_q}$ is a field. Suppose additionally that $R$
contains the field $\kkk(z)$. Let $f\in R$ and $a\in \C(z)$ be such
that $f$ is an invertible solution of
\begin{equation}\label{eq:exdiffeq}
\sigma_q(f) = af.
\end{equation}
Then $f$ is $\sigma_\zeta$-algebraically dependent over $\kkk(z)$ if
and only if
\begin{equation}\label{eq:aqbb}
\varphi(a)=\sigma_q(b)/b
\end{equation}
for some $0\ne b \in \C(z)$ and
$1\ne\varphi(x)=x^{n_0}{\sigma_\zeta(x)}^{n_1}\cdot\ldots\cdot
{\sigma_\zeta^{t-1}(x)}^{n_{t-1}}$.
\end{theorem}
\begin{proof}
If~\eqref{eq:aqbb} holds, then
$$
\sigma_q(\varphi(f)/b)=
\varphi(\sigma_q(f))/\sigma_q(b)=\varphi(af)/\sigma_q(b)=\varphi(a)\varphi(f)/\sigma_q(b)=\varphi(f)/b.
$$
Therefore,
$$
\varphi(f)/b=c\in R^{\sigma_q}=\kkk.$$
Thus,
$
\varphi(f)=c\cdot b\in \kkk(z)$,
which gives a $\Sigma_1$-algebraic dependence for $f$ over $\kkk(z)$.
First, note that $z$ is algebraically independent over
$\kkk$. Indeed, suppose that there is a relation
$$a_n\cdot z^n+a_{n-1}\cdot z^{n-1}+\ldots+a_0=0 $$ for some $a_i\in \kkk$. Applying
$\sigma_q$ $n$ times, we obtain the following system of linear equations
$$
\begin{pmatrix}
1&1&\ldots&1&1\\
q^n&q^{n-1}&\ldots&q&1\\
\vdots&\vdots&\ddots&\vdots&\vdots\\
%{(q^n)}^{n-1}&{(q^{n-1})}^{n-1}&\ldots&q^{n-1}&1\\
{(q^n)}^{n}&{(q^{n-1})}^{n}&\ldots&q^{n}&1\\
\end{pmatrix}
\begin{pmatrix}
a_n\cdot z^n\\
a_{n-1}\cdot z^{n-1}\\
\vdots\\
a_0\\
\end{pmatrix}
=0
$$
Since the matrix is invertible,  our relation is of the form $a\cdot z^k=0$
for some $a\in \kkk$. Since $\kkk$ is a field, we have $z^k=0$. However, $z\in\C(z)$, which is a contradiction.

Assume now that $f$ is $\Sigma_1$-algebraically dependent over
$\kkk(z)$. Let $C$ be the $\Sigma_1$-closure of $\kkk$ and $K$ be the total
ring of fraction of the polynomial ring $C[z]$, where
$\sigma_q(z)=qz$ and  $\sigma_\zeta(z)=\zeta z$. So, the field $\kkk(z)$
is naturally embedded into $K$. Let $D$ be the smallest
$\Sigma$-subring in $R$ generated by $\kkk(z)$, $f$, and $1/f$ and let
$$\m\subseteq K\otimes_{\kkk(z)}D$$ be a maximal $\Sigma$-ideal.
Then, $$L=\left(K\otimes_{\kkk(z)}D\right)/\m$$ is a Picard--Vessiot ring over
$K$ for equation~\eqref{eq:exdiffeq}. The image of $f$ in $L$ will
be denoted by $\bar f$. Since $f$ is $\Sigma_1$-algebraically
dependent over $\kkk(z)$, $\bar f$ is $\Sigma_1$-algebraically
dependent over $K$.

It follows from Section~\ref{sec:torsor} that $\bar f$ is
$\Sigma_1$-algebraically dependent over $K$ if and only the
$\Sigma$-Galois group $G$ of equation~\eqref{eq:exdiffeq} is a
proper subgroup of ${\Gm}_{,\Sigma_1}$. Then, by
Example~\ref{ex:subgroupsofGm}, there exists a multiplicative
$$
\varphi\in (F_{\Sigma_1}\Q)\{x,1/x\}_{\Sigma_1}
$$
(see also~\eqref{eq:defofFSigma}) such that $G$ is given by the
equation
$
\varphi(x) = 1$.
Therefore, for all $\phi$ in the Galois group,
$$
\phi(\varphi(\bar f))= \varphi(\phi(\bar f))=\varphi(c_\phi\cdot
\bar f) = \varphi(c_\phi)\cdot \varphi(\bar f)=1\cdot \varphi(\bar
f)= \varphi(\bar f).
$$
Hence, by Proposition~\ref{prop:LGK}, we have
$$
b:= \varphi(\bar f) \in K = C(z),$$
as in \cite[Proposition~3.1]{hardouin_differential_2008}. Since $f$
is invertible, $\bar f$ is also invertible and, since $\varphi$ is
multiplicative, $\varphi(\bar f)$ is invertible as well.  Therefore,
\begin{align}\label{eq:dep}
\varphi(a) = \varphi\left(\sigma_q(\bar f)/\bar f\right)=
\sigma_q\left(\varphi(\bar f)\right)/\varphi(\bar f)= \sigma_q(b)/b.
\end{align}
We will show now that $b$ can be chosen from $(F_{\Sigma_1}\C)(z)$
satisfying~\eqref{eq:aqbb} as in
\cite[Corollary~3.2]{hardouin_differential_2008}. We have the
equalities $a=\bar a/c$ and $b=\bar b /d$, where
$\bar a, c\in \C[z]$ and $\bar b, d\in C[z]$. Consider the
coefficients of $\bar b$ and $d$ as difference indeterminates. Then,
equation~\eqref{eq:dep} can be considered as a system of equations
in the coefficients of $\bar b$ and $d$. Indeed,
equation~\eqref{eq:dep} is equivalent to
$$
\varphi(\bar a/c)=\sigma_q{\left(\bar b/d\right)}\big/{\left(\bar b/d\right)}.
$$
So, we have
\begin{equation}\label{eq:phiabcd}
\varphi(\bar a)\cdot\sigma_q(d)\cdot\bar b-\varphi(c)\cdot\sigma_q(\bar b)\cdot d=0
\end{equation}
The left-hand side of equation~\eqref{eq:phiabcd} is a polynomial in
$z$. The desired system of equations is given by the equalities for
all coefficients.
Now note that the condition of $y\in C[z]$ being invertible in $C(z)$
is given by the inequation
$$
y\cdot\sigma_\zeta(y)\cdot\ldots\cdot\sigma_{\zeta}^{t-1}(y)\neq 0.
$$
Therefore, the coefficients of the polynomials $\bar b$ and $d$ are
given by the system of equations and inequalities. Since the
pseudofield $F_{\Sigma_1}\C$ is $\Sigma_1$-closed, existence of
invertible $\bar b$ and $d$ with coefficients in $C$ implies
existence of invertible $\bar b$ and $d$ with coefficients in
$F_{\Sigma_1}\C$ (see~\cite[Proposition~25~(3)]{Dima}).

 We will now show that $b\in\C(z)$ and $\varphi$ can be found of
the desired form. We have proven that
\begin{equation}\label{eq:overc}
\varphi(a)=\sigma_q(b)/b
\end{equation}
for some $b\in (F_{\Sigma_1}\C)(z)$. It follows from
Example~\ref{ex:subgroupsofGm} that
\begin{align*}
\varphi(x)=&e_0\cdot x^{n_{0,0}}\cdot\sigma_\zeta(
x)^{n_{0,1}}\cdot\ldots\cdot\sigma_\zeta^{t-1}(x)^{n_{0,t-1}}
+\ldots +\\
&+e_{t-1}\cdot
x^{n_{t-1,0}}\cdot\sigma_\zeta(x)^{n_{t-1,1}}\cdot\ldots\cdot\sigma_\zeta^{t-1}(x)^{n_{t-1,t-1}}.
\end{align*}
Note that if $a\in (F_{\Sigma_1}\C)(z)$ belongs to $\C(z)$, then
$$
\gamma_e(a)=a\quad \text{and}\quad
\gamma_e(\sigma_\zeta^i(a))=\sigma_\zeta^i(\gamma_e(a)),$$
%}
where the $\sigma_q$-homomorphism $\gamma_e\colon
(F_{\Sigma_1}\C)(z)\to \C(z)$ is defined in~\eqref{eq:gammamu}.
Applying this homomorphism to~\eqref{eq:overc}, we obtain
$$
a^{n_{0,0}}\cdot\sigma_\zeta(
a)^{n_{0,1}}\cdot\ldots\cdot\sigma_{\zeta}^{t-1}(a)^{n_{0,t-1}}=\sigma_q(\gamma_e(b))/\gamma_e(b),
$$
which concludes the proof.  \end{proof}

\subsection{Setup for meromorphic functions}
%\begin{example}
The ring of all meromorphic functions on $\mathbb C^*$ will be
denoted by $\mathcal M$. For any nonzero complex number $a$ we
define an automorphism $\sigma_a : \mathcal M \to \mathcal M$ by
$$\sigma_a(f)(z)=f(az).$$ Let $\Sigma_1\subseteq \mathbb C^*$ be a
finite subgroup. Then $\Sigma_1$ is a cyclic group generated by a
root of unity $\zeta$ of degree $t$.
Let $q\in\mathbb C$ be  such that $|q|>1$. Now, we
have an action of the group $\Sigma=\mathbb Z\oplus\mathbb
Z/t\mathbb Z$, where the first summand is generated by $\sigma_q$
and  the second one is generated by $\sigma_\zeta$.

The set of all $\sigma_q$-invariant meromorphic functions will be
denoted by $\kk$. As we can see $\kk$ is a $\Sigma_1$-ring. Let $C$ be
the $\Sigma_1$-closure of the field $\kk$. Supply the polynomial ring
$C[z]$ with the following structure of a $\Sigma$-ring: $$\sigma_q(z)=qz\quad \text{and}\quad
\sigma_{\zeta}(z)=\zeta z.$$ Let $K$ be the total ring of fractions of
$C[z]$, so, $K$ is a Noetherian $\Sigma$-pseudofield with
$\Sigma_1$-closed subpseudofield of $\sigma_q$-constants $C$.
The meromorphic function $z$ is algebraically independent over $\kk$.
Hence, the minimal $\Sigma$-subfield in $\mathcal M$ generated by $\kk$
and $z$ is the ring of rational functions $\kk(z)$. Thus, this field can
be naturally embedded into $K$ with $z$ being mapped to $z$.

\subsection{Jacobi's theta-function}\label{sec:Jacobi}
We will study $\Sigma_1$-relations for Jacobi's theta-function (being a solution of 
$\theta_q(qz) = -qz\cdot\theta_q(z)$)\begin{equation}\label{eq:thetafirst} \theta_q(z) = -\sum_{n \in\Z}(-1)^nq^\frac{-n(n-1)}{2}z^n,\quad z \in \C,
\end{equation}
with coefficients in $\kk(z)$.

\subsubsection{Relations for $\theta_q$ with $q$-periodic coefficients}\label{sec:431}
First,  we will show that there are many relations of such form:
\begin{enumerate}
\item Suppose that $t\geqslant 3$. Then, the
function
$$
\lambda=\theta_q(z)\cdot\theta_q^{-2}(\zeta z)\cdot\theta_q(\zeta^2
z)
$$
is $\sigma_q$-invariant. Therefore, $\theta_q$ vanishes the
following nontrivial $\Sigma_1$-polynomial:
$$
y\cdot\sigma_{\zeta^2}(y) - \lambda\cdot (\sigma_\zeta(y))^2\in
\kk(z)\{y\}.
$$
\item Suppose that $t=m\cdot n$,
where $m$ and $n$ are coprime. Then, there exist two numbers
$u\ne v$ such that the automorphisms $\sigma_{\zeta}^u\ne \sigma_{\zeta}^v$ but
$\sigma_{\zeta}^{un}=\sigma_{\zeta}^{vn}\neq \id$. Then, the function
$$
\lambda = \theta_q^n(\zeta^uz)\cdot\theta_q^{-n}(\zeta^vz)
$$
is $\sigma_q$-invariant. Therefore, $\theta_q$ vanishes the
following nontrivial $\Sigma_1$-polynomial:
$$
(\sigma_{\zeta^u}(y))^n - \lambda \cdot(\sigma_{\zeta^v}(y))^n\in
\kk(z)\{y\}.
$$
\item For any given $\zeta$, the function
$$
\lambda = \theta_q^{t}(z)\cdot\theta_q^{-t}(\zeta z)
$$
is $\sigma_q$-constant. Therefore, $\theta_q$ vanishes the following
nontrivial $\Sigma_1$-polynomial:
$$
y^t - \lambda\cdot (\sigma_\zeta(y))^t\in \kk(z)\{y\}.
$$
\end{enumerate}

\subsubsection{Periodic difference-algebraic independence for $\theta_q$ with $q$-periodic coefficients}

We will show now that in some sense these relations are the only possible ones.

\begin{lemma}\label{lemma1}
Suppose that for some rational function $b\in \kk(z)$ there is a
relation 
$$
(-qz)^{k_0}\left(-q\zeta
z\right)^{k_1}\cdot\ldots\cdot\left(-q\zeta^{t-1}z\right)^{k_{t-1}}=\sigma_q(b)/b
$$
for some $k_i\in \mathbb Z$. Then,
$
\sum\limits_{i=0}^{t-1}k_i =0$.
\end{lemma}
\begin{proof}
The function $\sigma_q(b)/b$ is of the following form
$
\sigma_q(b)/b=h/g$,
where $h$ and $g$ have the same degree and the same leading
coefficient. The equality  follows from the condition on the degree.
\end{proof}

\begin{lemma}\label{lemma2}
Suppose that there exist $\lambda\in \kk(z)$ and $\eta$, $q \in \C$ such that
$
\sigma_q(\lambda)=\eta\cdot\lambda$,
where $|\eta|=1$ and $0\ne q$ is not a root of unity. Then, $\lambda\in \kk$ and $\eta=1$.
\end{lemma}
\begin{proof}
Let 
$$
\lambda=a\cdot z^r\cdot
\frac{(z-a_1)\cdot\ldots\cdot(z-a_n)}{(z-b_1)\cdot\ldots\cdot(z-b_m)}
$$
be the irreducible representation of $\lambda$, where $a_i,b_i\in \overline\kk$, the algebraic closure of $\kk$. By
the hypothesis, we have
$$
q^{r+n-m}\cdot\frac{\left(z-a_1/q\right)\cdot\ldots\cdot\left(z-a_n/q\right)}{\left(z-b_1/q\right)\cdot\ldots\cdot\left(z-b_m/q\right)}
=\eta\cdot\frac{(z-a_1)\cdot\ldots\cdot(z-a_n)}{(z-b_1)\cdot\ldots\cdot(z-b_m)}.
$$
Therefore, $q^{r+n-m}=\eta$. Thus, $r+n-m=0$ and $\eta=1$. Moreover, the
sets
$
\{a_1,\ldots,a_n\}$ and $\left\{a_1/q,\ldots,a_n/q\right\}$
must coincide.
If $\lambda\notin \kk$, then, from  $r+n-m=0$, it follows that
either $n>0$ or $m>0$. Suppose that the first inequality holds. There exists $i$
such that $a_1=\frac{a_i}{q}$. If $i=1$, then we set $i_0=1$. Otherwise, $i>1$ and, rearranging the
elements $\{a_j\}$ for $j>1$ suppose that $i=2$. Again, $a_2=\frac{a_i}{q}$. If $i=1$, we
set $i_0=i$. Otherwise, $i>2$ and rearranging the elements $\{a_j\}$ for
$j>2$, suppose that $i=3$, and so on. Since there are only
finitely many elements, the process will stop and we obtain a number
$i_0$ with the following system of equations:
$$
%\left\{\begin{aligned}
a_1=a_2/q,\
a_2=a_3/q,\
\ldots,\
a_{i_0}=a_1/q
%\end{aligned}
%\right.
$$
Therefore, $q^{i_0}=1$. Thus, $|q|=1$, which is a contradiction.
\end{proof}

\begin{proposition}\label{prop:homo}
Let the pseudofield $K$ be as above. Let $R$ be a Picard--Vessiot
ring over $K$ for the equation
$\sigma_q(y)=-qz\cdot y$
and $L$ be the corresponding Picard--Vessiot pseudofield. Suppose
 $f$ is an invertible solution in $R$. Then $L\otimes_K R$ is a
graded ring such that $f$ is of degree $1$ and $\sigma_q$ and
$\sigma_\zeta$ preserve the grading.
\end{proposition}
\begin{proof}
It follows from Proposition~\ref{prop:isomCG}  that
$R\otimes_K R=R\otimes_C C\{G\}$, where $G$ is the
corresponding Galois group. Multiplying by $L\mathop{\otimes}_R -$,
we obtain: $L\otimes_K R=L\otimes_CC\{G\}$. Since
group $G$ is a subgroup of $\G_m$,
$$C\{G\}=C\{x,1/x\}_{\Sigma_1}/J,$$ where the ideal $J$ is generated
by difference polynomials of the form
$$e_0\cdot x^{k_0}\cdot\left(\sigma_\zeta
x\right)^{k_1}\cdot\ldots\cdot(\sigma_\zeta^{t-1}x)^{k_{t-1}}-e_0$$
(see Example~\ref{ex:subgroupsofGm} for details). The ring
$C\{x,1/x\}_{\Sigma_1}$ is a graded ring such that $x$ is
homogeneous of degree $1$ and $\sigma_\zeta$ preserves the grading.
In the proof of Theorem~\ref{thm:app}, we have obtained that
$$
{(-qz)}^{k_0}\cdot {\left(-q\zeta
z\right)}^{k_1}\cdot\ldots\cdot\left(-q\zeta^{t-1}
z\right)^{k_{t-1}} = \sigma_q(b)/b
$$
for some $b\in \C(z)$. Thus, it follows from Lemma~\ref{lemma1} that
$$\sum_{i=0}^{t-1}k_{i}=0.$$  Therefore, the ideal $J$ is homogeneous.
Hence, $C\{G\}$ is graded. Thus, $L\mathop{\otimes}_C C\{G\}$ is
graded. Since $f=\bar f\cdot y$, where $\bar f\in L$ is a solution
of the equation in $L$, then $f$ is a homogeneous element of degree
$1$. Since $x$ is $\sigma_q$-constant, $\sigma_q$ preserves the
grading.
\end{proof}

\begin{theorem}\label{thm:Benrelation}
%Let the pseudofield $K$ be as above and suppose additionally that
%$t$ is a prime number. Let $R$ be a Picard--Vessiot ring over $K$ for
%the equation $\sigma_q(y)=-qz\cdot y$. Then, 
For every prime number $t$, every relation of the
form
\begin{equation}\label{eq:Ben}
\lambda_0 +
\sum_{d=1}^{t-1}\left(\lambda_{0d}\cdot\theta_q{(z)}^d+\lambda_{1d}\cdot\theta_q{\left(\zeta
z\right)}^d+\ldots+\lambda_{t-1 d}\cdot\theta_q{\left(\zeta^{t-1}z\right)}^d\right)=0,
\end{equation}
with $\lambda_0,\lambda_{ij}\in \kk(z)$, implies that
$\lambda_0=\lambda_{ij}=0$.
\end{theorem}
\begin{proof}
Let the pseudofield $K$ be as above, $R$ be a Picard--Vessiot ring over $K$ for
the equation $\sigma_q(y)=-qz\cdot y$, and
 $L$ be the corresponding Picard--Vessiot pseudofield for $R$.
 It follows  from Proposition~\ref{prop:homo} that $D=L\otimes_K R$ is
a graded ring such that the image of  $\theta_q$ in $D$ is
homogeneous of degree $1$. Suppose now that $\theta_q$ satisfies an
equation of the form~\eqref{eq:Ben}. Then, the same equation holds
in $R$ and, after embedding $R$ into $D$, it holds in $D$. Since $D$
is graded, our equation is homogeneous. Thus, it is of the form
$$
\lambda_{0}\cdot\theta_q{(z)}^d+\lambda_{1}\cdot\theta_q{\left(\zeta
z\right)}^d+\ldots+\lambda_{t-1}\cdot\theta_q{\left(\zeta^{t-1}z\right)}^d=0
$$
for some $d$. Consider the shortest equation and rewrite it as
follows
$$
\theta_q{(z)}^d+\lambda_{r}\cdot\theta_q{\left(\zeta^r z\right)}^d+\ldots+\lambda_{t-1
}\cdot\theta_q{\left(\zeta^{t-1}z\right)}^d=0,
$$
where $\lambda_{r}\cdot\theta_q{\left(\zeta^r z\right)}^d$ is the first nonzero summand
immediately following $\theta_q(z)^d$. Applying $\sigma_q$ and dividing by $(-qz)^d$,
we obtain
$$
\theta_q{(z)}^d+\sigma_q(\lambda_{r})\cdot{(\zeta^r)}^d\cdot\theta_q\left(\zeta^r
z\right)^d+\ldots+\sigma_q(\lambda_{t-1
})\cdot\left(\zeta^{t-1}\right)^d\cdot\theta_q{\left(\zeta^{t-1}z\right)}^d=0
$$
Therefore, $$\sigma_q(\lambda_r)=\zeta^{-rd}\cdot\lambda_r.$$ Now, it follows from
Lemma~\ref{lemma2}  that $\zeta^{-rd}=1$, contradiction. Thus, $\theta_q{(z)}^d=0$ must hold, but Picard--Vessiot
pseudofield is reduced, which is a contradiction again.
\end{proof}

%\end{example}

\subsubsection{Difference-algebraic independence for $\theta_q$ over $\C(z)$}

We will now show difference-algebraic independence for $\theta_q$ over $\C(z)$.

\begin{example}\label{ex:thetarational} Consider an equation
$$
F(\theta_q)=\sum\nolimits_{(n_1,\ldots,n_p)\in\Z^p}
g_{n_1,\ldots,n_p}(z)\cdot \theta_q(\alpha_1
z)^{n_1}\cdot\ldots\cdot\theta_q(\alpha_p z)^{n_p}=0,
$$
where $g_{n_1,\ldots,n_p}\in\C(z)$ and $1\neq\alpha_i\neq\alpha_j
$ in $\C^*/q^{\Z}$. We will show that all
$g_{n_1,\ldots,n_p}$ are equal to zero. Since the sum is finite, there exists a monomial
$$
M(\theta_q)=\theta_q(\alpha_1 z)^{d_1}\cdot \ldots\cdot\theta_q(\alpha_p
z)^{d_p}
$$
such that $M(\theta_q)\cdot F(\theta_q)$ contains monomials with negative
powers. Now, we will calculate the poles of a given monomial with
negative powers. The poles of the $i$-th factor of the monomial
$$
M(\theta_q)=\frac{1}{\theta_q(\alpha_1
z)^{n_1}}\cdot\ldots\cdot\frac{1}{\theta_q(\alpha_p z)^{n_p}}.
$$
 are $\alpha_i^{-1}q^{r}$ for all $r\in
\Z$ and the multiplicity of each of the poles is $n_i$. The poles of
distinct factors are distinct. Indeed, suppose that
$$\alpha_i^{-1}\cdot q^{r_1}=\alpha_j^{-1}\cdot q^{r_2}.$$ Then,
$\alpha_j=\alpha_i\cdot q^{r_2-r_1}$. Therefore, $\alpha_i=\alpha_j$ in
$\C^*/q^{\mathbb Z}$, which is a contradiction. Thus, the set of all poles
of the monomial $M(\theta_q)$ is $\alpha_i^{-1}\cdot q^r$ with
multiplicity $n_i$.

Every function $g\in \C(z)$ has only finitely many poles and
zeros, so, all of them are inside of a disk $U_d=\{z\in \C\mid
|z|<d\}$. So, the set of all poles for $M(\theta_q)$ and $g\cdot
M(\theta_q)$ coincides in $\C\setminus U_d$ for some $d$. There exists a disk $U_d$ such that this property holds for all summands
in $F$. We can rewrite $F$ as follows:
\begin{align*}
F(\theta_q)&=\sum\nolimits_{n_1} \left(\sum\nolimits_{n_2,\ldots,n_p}
g_{n_1,\ldots,n_p}\cdot\theta_q(\alpha_1 z)^{n_1}\cdot\ldots\cdot\theta_q(\alpha_p
z)^{n_p}\right)=\\
&=\sum\nolimits_{n_1}F_{n_1}(\theta_q)=0.
\end{align*}
The point $\alpha_1^{-1}q^{r_1}$ (where $r_1$ is large enough positive
if $q>1$ and large enough negative if $q<1$) is a pole for all
summands $F_{n_i}$ and the multiplicity of this pole is different for
different $n_i$. To annihilate these poles,
$F_{n_1}=0$ must hold for all $n_1$. Repeating the same argument for all
$n_i$, we arrive at
$$
g_{n_1,\ldots,n_p}(z)\cdot
\theta_q(\alpha_1 z)^{n_1}\cdot\ldots\cdot\theta_q(\alpha_p z)^{n_p}=0
$$
for each $n_1,\ldots,n_p$. Therefore, $g_{n_1,\ldots,n_p}=0$.

It follows from this result that for an arbitrary root of unity
$\zeta$ the function $\theta_q$ is $\sigma_\zeta$-algebraically
independent over $\mathbb C(z)$ in the field of meromorphic
functions on $\mathbb C^*$.
However,  to generalize this result to finitely many roots of
unity, we need to require the following: $$\text{for all $i$ and $j$}\quad \zeta_i^k=\zeta_j^m\quad \text{implies}\quad
\zeta_i^k=\zeta_j^m=1.$$ Otherwise, the result is not true. Indeed,
if $\zeta_i^k=\zeta_j^m\neq1$, then the relation
$$\sigma_{\zeta_i}^k(\theta_q)-\sigma_{\zeta_j}^m(\theta_q)=0$$
is non-trivial. Indeed, note that  the difference
indeterminates $\sigma_{\zeta_i}^kx$ and $\sigma_{\zeta_j}^mx$ are distinct even in the
difference polynomial ring $\Q\{x\}_{\Sigma_1}$ in spite of the
fact that they define the same automorphisms of meromorphic
functions.
\end{example}

\subsection{General order one $q$-difference equations}\label{sec:generalqdifference}
We will start by discussing several examples of $\sigma_\zeta$-dependence and independence and finish by providing a general criterion
in Theorem~\ref{thm:6}. 
%\subsubsection{Initial examples}
\begin{example} For $a(z) = \frac{z+1}{z-1}$, $t=2$, and $\zeta = -1$, we have
$$
\sigma_\zeta(a)(z)\cdot\sigma_\zeta^0(a)(z) = \frac{-z+1}{-z-1}\cdot\frac{z+1}{z-1} = 1 = \sigma_q(1)/1.
$$
Let $g$ be a meromorphic function on
$\C\setminus\{0\}$ such that $\sigma_q(g)=\frac{z+1}{z-1}\cdot g$. Then,
$g(z)\cdot g(-z)$ is $\sigma_q$-invariant:
$$
\sigma_q(g\cdot\sigma_\zeta(g))=\frac{z+1}{z-1}\cdot g\cdot\sigma_\zeta\left(\frac{z+1}{z-1}\cdot g\right)=\frac{z+1}{z-1}\cdot\frac{-z+1}{-z-1}\cdot g\cdot \sigma_\zeta(g)=g\cdot \sigma_\zeta(g).
$$
So, the function $g$ is $\sigma_\zeta$-algebraically dependent over
$\kk$.
\end{example}

\begin{example}
For $a(z) = 2^z$ and $t=4$ with $\zeta=i$, we have
$$
\sigma_\zeta^2(a)(z)\cdot\sigma_\zeta^0(a)(z) = 2^{-z}\cdot2^z=1=\sigma_q(1)/1.
$$
As before, let $g$ be a meromorphic function on
$\C\setminus\{0\}$ such that $\sigma_q(g)=2^z\cdot g$. Then, $g(z)\cdot g(-z)$ is
$\sigma_q$-invariant. Indeed,
%\vspace{-0.1in}
$$
\sigma_q(g\cdot\sigma_\zeta^2(g))=2^z\cdot g\cdot\sigma_\zeta^2\left(2^z\cdot g\right)=2^z\cdot
2^{-z}\cdot g\cdot\sigma_\zeta^2(g)=g\cdot\sigma_\zeta^2(g).
$$
So, the function $g$ is $\sigma_\zeta$-algebraically dependent over
$\kk$.
\end{example}

\subsubsection{General characterization of periodic difference-algebraic independence}
Let $a\in \C(z)$ and $q,\zeta\in \C^*$ be such that
$|q|>1$ and $\zeta$ is a primitive root of unity of order $t$.
Then, $a$ can
be represented as follows
$$
a = \lambda\cdot z^T\cdot
\prod_{k=0}^{t-1}\prod_{d=-N-1}^{N}\prod_{i=1}^{R}\left( z - \zeta^k\cdot
q^d \cdot r_i \right)^{s_{k,d,i} },
$$
where $\lambda, r_i\in \mathbb C^*$ and the $r_i$'s are distinct in
$\mathbb C^*\big/\zeta^{\mathbb Z}\cdot q^{\mathbb Z}$. Let
%\vspace{-0.1in}
$$
a_{i,k} = \sum_{d=-N-1}^{N} s_{k,d,i},\ 
d_{k,i} = \sum_{j=0}^{t-1} \zeta^{k\cdot j}\cdot a_{i,j}
,\ \text{and}\ 
D=
\begin{pmatrix}
d_{0,1}&d_{0,2}&\ldots&d_{0,R}\\
d_{1,1}&d_{1,2}&\ldots&d_{1,R}\\
%d_{2,1}&d_{2,2}&d_{2,3}&\ldots&d_{2,R}\\
\vdots&\vdots&\ddots&\vdots\\
d_{t-1,1}&d_{t-1,2}&\ldots&d_{t-1,R}\\
\end{pmatrix}.
$$
The following result combined with Theorem~\ref{thm:app} provides  a complete characterization of all equations~\eqref{eq:exdiffeq} whose solutions are $\sigma_\zeta$-algebraically independent.

\begin{theorem}\label{thm:6}
Let $a\in \C(z)$ and $D$
be as above. Then,
\begin{enumerate}
\item If $T=0$ and, either $\lambda^\mathbb Z\cap q^\mathbb Z\neq1$ or $\lambda$ is a root of unity, then
there exist  $b\in \C(z)$ and a
multiplicative function
$$
\varphi(x)=x^{n_0}\cdot \left(\sigma_\zeta
x\right)^{n_1}\cdot\ldots\cdot(\sigma_\zeta^{t-1} x)^{n_{t-1}}
$$
such that $\varphi(a)=\sigma_q(b)/b$ if and only if the matrix $D$
contains a zero row.
\item If either $T\neq0$ or, $\lambda^\mathbb Z\cap q^\mathbb Z=1$ and $\lambda$ is not a root of unity, then
there exist  $b\in \mathbb C(z)$ and a
multiplicative function
$$
\varphi(x)=x^{n_0}\cdot \left(\sigma_\zeta
x\right)^{n_1}\cdot\ldots\cdot (\sigma_\zeta^{t-1} x)^{n_{t-1}}
$$
such that $\varphi(a)=\sigma_q(b)/b$ if and only if  $D$
contains a zero row other than the first one.
\end{enumerate}
\end{theorem}
\begin{proof}
We will write $\varphi$ and $b$ with undetermined coefficients and
exponents. Suppose that
$$
b=\mu\cdot z^M\cdot
\prod_{k=0}^{t-1}\prod_{d=-N}^{N}\prod_{i=1}^{R}\left( z - \zeta^k\cdot
q^d\cdot r_i \right)^{l_{k, d, i}}
\ 
\text{and}\ 
\varphi(x)=x^{n_0}\cdot \left(\sigma_\zeta x\right)^{n_1}\cdot\ldots\cdot
(\sigma_\zeta^{t-1} x)^{n_{t-1}}
$$
are such that  $\varphi(a)=\sigma_q(b)/b$. Let us
calculate the right and left-hand sides of this equality. We
see that
$$
\sigma_q(b)=\mu \cdot q^{M+\sum_{k, d, i} l_{k, d, i}} \cdot z^M\cdot
\prod_{k=0}^{t-1} \prod_{d=-N}^{N}\prod_{i=1}^{R} \left( z - \zeta^k\cdot
q^{d - 1}\cdot r_i \right)^{l_{k, d, i}}.
$$
Hence,
$$
\begin{aligned}
&\frac{\sigma_q(b)}{b}=q^{M+\sum_{k, d, i} l_{k, d, i}} \cdot
\prod_{k=0}^{t-1} \prod_{d=-N-1}^{N-1}\prod_{i=1}^{R} \left( z -
\zeta^k\cdot
q^{d}\cdot r_i \right)^{l_{k, d+1, i}}\cdot\\
&\qquad\qquad\qquad\cdot
\prod_{k=0}^{t-1} \prod_{d=-N}^{N}\prod_{i=1}^{R} \left( z -
\zeta^k\cdot q^{d}\cdot r_i \right)^{-l_{k, d, i}}=\\
&=q^{M+\sum_{k, d, i} l_{k, d, i}} \cdot
\prod_{k=0}^{t-1}\prod_{i=1}^{R} \Bigl[\left( z - \zeta^k\cdot
q^{-N-1}\cdot r_i \right)^{l_{k, -N, i}}\cdot\\
&\qquad\qquad\qquad\qquad\qquad\cdot\prod_{d=-N}^{N-1} \left( z - \zeta^k\cdot q^{d}\cdot r_i \right)^{l_{k,
d+1, i}-l_{k, d, i}}\left( z - \zeta^k\cdot q^{N}\cdot r_i
\right)^{-l_{k, N, i}}\Bigr].
\end{aligned}
$$
Now, we calculate the left-hand side. We see that
$$
\begin{aligned}
\sigma_\zeta^r a &= \lambda\cdot\zeta^{rT+\sum_{k, d, i} r\cdot s_{k,
d, i}} \cdot z^T\cdot
\prod_{k=0}^{t-1}\prod_{d=-N-1}^{N}\prod_{i=1}^{R}\left( z -
\zeta^{k-r}\cdot q^d \cdot r_i \right)^{s_{k, d, i} }=\\
&=\lambda\cdot\zeta^{rT+\sum_{k, d, i} r\cdot s_{k, d, i}} \cdot
z^T\cdot \prod_{k=0}^{t-1}\prod_{d=-N-1}^{N}\prod_{i=1}^{R}\left( z -
\zeta^{k}\cdot q^d \cdot r_i \right)^{s_{k+r, d, i} }.
\end{aligned}
$$
Hence,
$$
\begin{aligned}
\varphi(a)&=\lambda^{\sum_{r=0}^{t-1} n_r}\cdot\zeta^{\left(T+\sum_{k,
d, i} s_{k, d, i}\right)\cdot\left( \sum_{r=0}^{t-1}r\cdot n_r \right)} \cdot
z^{T\cdot\left(\sum_{k=0}^{t-1} n_r\right)}\cdot\\
&\qquad\cdot
\prod_{k=0}^{t-1}\prod_{d=-N-1}^{N}\prod_{i=1}^{R}\left( z -
\zeta^{k}\cdot q^d \cdot r_i \right)^{\sum_{r=0}^{t-1} n_r s_{k+r, d, i}
}.
\end{aligned}
$$
Now, the equation $\varphi(a)=\sigma_q(b)/b$ gives the following
system of equations
$$
\left\{
\begin{aligned}
&\left\{
\begin{aligned}
\sum_{r=0}^{t-1} s_{k+r, -N-1, i}\cdot n_r &= l_{k, -N, i}\\
\sum_{r=0}^{t-1} s_{k+r, d, i}\cdot n_r &= l_{k, d+1, i } - l_{k, d, i},\quad \quad{-N\leqslant d\leqslant N-1}\\
\sum_{r=0}^{t-1} s_{k+r, N, i} \cdot n_r &= -l_{k, N, i}\\
\end{aligned}\right.\\
&\lambda^{\sum_{k=0}^{t-1}n_r}\cdot\zeta^{\left(T+\sum_{k, d, i} s_{k, d, i}\right)\cdot\left(\sum_{r=0}^{t-1} r\cdot n_r\right)} = q^{M+\sum_{k, d, i} l_{k, d, i}}\\
&T\cdot\sum_{r=0}^{t-1} n_r=0
\end{aligned}\right.
$$
In this system, the unknown variables are $l_{k, d, i}$, $n_r$, and $M$. If
$l_{k, d, i}, n_r,  M$ is a solution of the system such that not
all $n_r$'s are zeroes, then $t\cdot l_{k, d, i}, t\cdot n_r, t\cdot M$ is a solution
with the same property. Therefore, we can replace the second equation with
$$
\lambda^{\sum_{k=0}^{t-1}n_r}= q^{M+\sum_{k, d, i} l_{k, d, i}}.$$
The first subsystem can be rewritten as follows:
$$
\begin{pmatrix}
s_{k, -N-1, i}&s_{k+1, -N-1, i}&\ldots&s_{k-1, -N-1, i}\\
s_{k, -N, i}&s_{k+1, -N, i}&\ldots&s_{k-1, -N, i}\\
\vdots&\vdots&\ddots&\vdots\\
s_{k, N, i}&s_{k+1, N, i}&\ldots&s_{k-1, N, i}\\
\end{pmatrix}
\begin{pmatrix}
n_0\\
n_1\\
\vdots\\
n_{t-1}
\end{pmatrix}
=
\begin{pmatrix}
l_{k, -N, i}\\
l_{k, -N +1, i} - l_{k, -N, i}\\
\vdots\\
-l_{k, N, i}\\
\end{pmatrix}.
$$
This system has a solution in $l_{k, d, i}$ if and only if the sum
of all equations is zero. Thus, we can replace this system with the
following:
$$
n_0\cdot \sum_{d=-N-1}^{N} s_{k,d,i}+n_1\cdot \sum_{d=-N-1}^{N}s_{k+1,d,i}+\ldots+n_{t-1}\cdot\sum_{d=-N-1}^{N}s_{k-1, d, i}=0.
$$
%$$
%\begin{pmatrix}
%\sum\limits_{d=-N-1}^{N} s_{k,d,i}&\sum\limits_{d=-N-1}^{N}s_{k+1,d,i}&\ldots&\sum\limits_{d=-N-1}^{N}s_{k-1, d, i}\\
%\end{pmatrix}
%\begin{pmatrix}
%n_0\\
%n_1\\
%\vdots\\
%n_{t-1}\\
%\end{pmatrix}
%=0
%$$
Using the definition of the $a_{i,j}$'s, we obtain the following
system:
$$
\begin{pmatrix}
a_{i,0}&a_{i,1}&\ldots&a_{i,t-1}\\
a_{i,1}&a_{i,2}&\ldots&a_{i,0}\\
\vdots&\vdots&\ddots&\vdots\\
a_{i,t-1}&a_{i,0}&\ldots&a_{i,t-2}\\
\end{pmatrix}
\begin{pmatrix}
n_0\\
n_1\\
\vdots\\
n_{t-1}\\
\end{pmatrix}
=0.
$$
Thus, for some integers $\gamma_{k,d,i,j}$, we have:
$$
\left\{
\begin{aligned}
&\begin{pmatrix}
a_{i,0}&a_{i,1}&\ldots&a_{i,t-1}\\
a_{i,1}&a_{i,2}&\ldots&a_{i,0}\\
\vdots&\vdots&\ddots&\vdots\\
a_{i,t-1}&a_{i,0}&\ldots&a_{i,t-2}\\
\end{pmatrix}
\begin{pmatrix}
n_0\\
n_1\\
\vdots\\
n_{t-1}\\
\end{pmatrix}
=0\\
&\lambda^{\sum_{k=0}^{t-1}n_r} = q^{M+\sum_{k, d, i} l_{k, d, i}}\\
&T\cdot\sum_{r=0}^{t-1} n_r=0\\
&l_{k, d, i} = \sum_{r=0}^{t-1} \gamma_{k, d, i, r}\cdot n_r
\end{aligned}\right.
$$
Consider the {\bf first case}: $T=0$ and $\lambda^\mathbb Z \cap
q^\mathbb Z\neq 1$. Then, for some $u,v\in \mathbb Z\setminus \{0\}$ we have
$\lambda^u=q^v$. Hence, the second equation is equivalent to
%\vspace{-0.1in}
$$
v\cdot\sum_{r=0}^{t-1} n_r = u\cdot \left(M +\sum\nolimits_{k, d, i} l_{k,
d, i}\right)
$$
Suppose that $n_r$ and $l_{k,d,i}$ form a solution of all
equations except for the second one, where not all $n_r$'s are zero. Then,
%\vspace{-0.15in}
$$
u\cdot n_r,\quad u\cdot l_{k,d,i},\quad M= \sum_{r=0}^{t-1} (v\cdot
n_r)-\sum\nolimits_{k, d, i} (u\cdot l_{k, d, i})
$$
form a solution of the whole system. Thus, in this case, we may exclude the
second equation.
Now we will check the case $T=0$ and $\lambda^w=1$ for some $w\in
\mathbb Z\setminus\{0\}$. In this situation, if $n_r, l_{k,d,i}$ is a
solution of all equations except for the second one, then
%\vspace{-0.1in}
$$
w\cdot n_r,\ \ w\cdot l_{k,d,i},\ \ M= - \sum\nolimits_{k,d,i} w\cdot l_{k,d,i}
$$
is a solution of the whole system.
Therefore, in this case, the existence of $\varphi$ and $b$ is
equivalent  to  % the condition that the systems
\begin{equation}\label{eq:firstcase}
\begin{pmatrix}
a_{i,0}&a_{i,1}&\ldots&a_{i,t-1}\\
a_{i,1}&a_{i,2}&\ldots&a_{i,0}\\
\vdots&\vdots&\ddots&\vdots\\
a_{i,t-1}&a_{i,0}&\ldots&a_{i,t-2}\\
\end{pmatrix}
\begin{pmatrix}
n_0\\
n_1\\
\vdots\\
n_{t-1}\\
\end{pmatrix}
=0
\end{equation}
having a nontrivial common solution.

Consider the {\bf second case}: $T\neq 0$ or
$\left(\lambda^\mathbb Z\cap q^\mathbb Z=1\ \text{and}\ \lambda\ \text{is not a root of unity}\right)$. If $T\neq 0$, then the third equation gives
$\sum_{r=0}^{t-1} n_r=0$, and if $\lambda^\mathbb Z\cap q^\mathbb
Z=1$ and $\lambda$ is not a root of unity, then the second equation
gives
$\sum_{r=0}^{t-1} n_r=0$. Therefore, in both cases, the second equation is of the
form
$$
M+\sum\nolimits_{k,d,i} l_{k,d,i}=0.$$
Again, if $n_r,~l_{k,d,i}$ form a solution of all equations except
the second one, where not all $n_r$'s are zeroes, then $$n_r,\ \ l_{k,d,i},\ \
M= - \sum\nolimits_{k,d,i} l_{k,d,i}$$ form a solution of the whole system with
the same property. Thus, in this case, we need to show the existence
of a nontrivial solution of the system
\begin{equation}\label{eq:secondcase}
\left\{
\begin{aligned}
&\begin{pmatrix}
a_{i,0}&a_{i,1}&\ldots&a_{i,t-1}\\
a_{i,1}&a_{i,2}&\ldots&a_{i,0}\\
\vdots&\vdots&\ddots&\vdots\\
a_{i,t-1}&a_{i,0}&\ldots&a_{i,t-2}\\
\end{pmatrix}
\begin{pmatrix}
n_0\\
n_1\\
\vdots\\
n_{t-1}\\
\end{pmatrix}
=0,\\
&
n_0 + n_1+\ldots +n_{t-1}
=0.
\end{aligned}\right.
\end{equation}
Since all the coefficients in~\eqref{eq:firstcase} and~\eqref{eq:secondcase}
are integers,  there is a nontrivial solution with integral
coefficients if and only if there is a nontrivial solution with
complex coefficients. Define
\begin{align*}
E_+&=
\begin{pmatrix}
1&1&1&\ldots&1\\
1&\zeta&\zeta^2&\ldots&\zeta^{t-1}\\
1&\zeta^2&\zeta^{2\cdot 2}&\ldots&\zeta^{2\cdot(t-1)}\\
\vdots&\vdots&\vdots&\ddots&\vdots\\
1&\zeta^{t-1}&\zeta^{(t-1)\cdot 2}&\ldots&\zeta^{(t-1)\cdot(t-1)}\\
\end{pmatrix},\\
%\\
E_-&=
\begin{pmatrix}
1&1&1&\ldots&1\\
1&\zeta^{-1}&\zeta^{-2}&\ldots&\zeta^{-(t-1)}\\
1&\zeta^{-2}&\zeta^{-2\cdot 2}&\ldots&\zeta^{-2\cdot(t-1)}\\
\vdots&\vdots&\vdots&\ddots&\vdots\\
1&\zeta^{-(t-1)}&\zeta^{-(t-1)\cdot 2}&\ldots&\zeta^{-(t-1)\cdot(t-1)}\\
\end{pmatrix},
\end{align*}
$$
A_{i}=
\begin{pmatrix}
a_{i, 0}&a_{i, 1}&\ldots&a_{i, t-1}\\
a_{i, 1}&a_{i, 2}&\ldots&a_{i, 0}\\
%a_{i, 2}&a_{i, 3}&a_{i, 4}&\ldots&a_{i, 1}\\
\vdots&\vdots&\ddots&\vdots\\
a_{i, t-1}&a_{i, 0}&\ldots&a_{i, t-2}\\
\end{pmatrix},
\quad
D_i=
\begin{pmatrix}
d_{0, i}&0&\ldots&0\\
0&d_{1, i}&\ldots&0\\
\vdots&\vdots&\ddots&\vdots\\
0&0&\ldots&d_{t-1, i}\\
\end{pmatrix}.
$$
A straightforward calculation shows that
$
E_+\cdot A_i = D_i\cdot E_-$.
Let $n$ be the vector with the coordinates $n_0,n_1,\ldots,n_{t-1}$.
Hence, in the {\bf first case},  the systems
$
E_+ \cdot A_i \cdot  n = D_i\cdot E_-\cdot n=0
$
have a nontrivial solution. This is equivalent to the condition that the
systems
$
D_i \cdot m = 0$
have a nontrivial solution, where $m = E_-\cdot n$. Since the $D_i$'s are
diagonal,  there is a common solution of all systems $D_i\cdot m
= 0$ if and only if the matrices $D_i$ have a zero in the same
place. In other words, there is an integer $i_0$ such that for all
$i$ we have $d_{i_0, i}=0$. The latter condition is equivalent to the
condition that there is a zero row in the matrix $D$.

Consider the {\bf second case}. Let $l=(1,1,\ldots, 1)$ with
$t$ coordinates. We must to show that the systems
$$
A_i \cdot n=0,\quad l\cdot n=0
%\left\{
%\begin{aligned}
%&A_i \cdot n=0\\
%& l\cdot n=0
%\end{aligned}\right.
$$
have a nontrivial solution. Multiplying by $E_+$, we have
%\left\{
%\begin{aligned}
$$
D_i \cdot E_- \cdot n=0,\quad l\cdot n=0.
%\end{aligned}\right.
$$
Let $p_1,\ldots, p_u$ be the positions of all zero rows in the matrix $D$.
And let $E_1,\ldots,E_u$ be the columns in $E_-^{-1}$ with the $p_i$'s as indices. Since the matrices $D_i$ are diagonal, every common solution of the systems $D_i \cdot E_-\cdot
n=0$ is of the form: $$n=W\cdot\mu,\quad W :=(E_1,\ldots,E_u),\quad \mu :=(\mu_1,\ldots,\mu_n)^T.$$
Then, the equation $l\cdot n = 0$ gives
$
l\cdot W\cdot \mu = 0
$
Now, we find a condition when $l\cdot E_i$ is zero. For this, note
that
$
(1,0,\ldots,0)\cdot E_-=(1,1,\ldots,1)$
and, therefore, $$
(1,1,\ldots,1)\cdot E_-^{-1}=(1,0,\ldots,0)
.$$
Hence, only the first column of the matrix $E_-^{-1}$ gives nonzero
elements in the vector $l\cdot W$. The system $l\cdot W \cdot \mu =
0$ has only the zero solution if and only if $W$ is just one column and
 $l\cdot W\neq 0$. Thus, this system has a nontrivial solution
if and only if $W$ contains a row of $E_-^{-1}$ other than the first
one. In other words, the elements $d_{k,i}$ are zeroes for some
$k\neq 0$ and all $i$, $1\Le i\Le R$. This is equivalent to the
condition that $D$ contains a zero row other than the first one.
\end{proof}

\begin{corollary}\label{cor:pairwise}
    In the situation of Theorem~\ref{thm:app},
    if the zeros and poles of $a\in\C(z)$ are pair-wise distinct modulo the group generated by $\zeta$ and $q$, then any non-zero solution $f$ of the equation
  % \begin{equation*}
       $ \sigma_q(f)=af$
  %  \end{equation*}
    is $\sigma_\zeta$-independent over $\kk(z)$.
\end{corollary}

\begin{example} If $0 \ne c\in \C$, then any non-zero solution $f$ of  $\sigma_q(f) = (z-c)\cdot f$ is $\sigma_\zeta$-independent over $\kk(z)$.
\end{example}

\section*{Acknowledgments} We are grateful to  Yves Andr\'e, Henri Gillet, Sergey Gorchinskiy, Charlotte Hardouin, Manuel Kauers, Alexander Levin, Alice Medvedev, Eric Rosen, Jacques Sauloy, Michael Singer, Lucia di~Vizio,  Michael Wibmer, and the referee for their helpful comments. 

\section*{Funding}B.~Antieau was supported by the NSF  Grants DMS-0901373 and CCF-0901175. A.~Ovchinnikov was supported by the grants: NSF  CCF-0901175, 0964875, and 0952591 and  PSC-CUNY  No.~60001-40~41.

\setlength{\bibsep}{1.5pt}
\bibliographystyle{abbrvnat}\small
\bibliography{difference_difference_equations}

\end{document}